\documentclass[english]{amsart}
\usepackage[latin9]{inputenc}
\usepackage{amsthm}
\usepackage{amssymb}
\usepackage{esint}

\makeatletter

\newcommand{\noun}[1]{\textsc{#1}}

\numberwithin{equation}{section}
\numberwithin{figure}{section}
  \theoremstyle{remark}
  \newtheorem*{rem*}{\protect\remarkname}
\theoremstyle{plain}
\newtheorem{thm}{\protect\theoremname}
  \theoremstyle{definition}
  \newtheorem{defn}[thm]{\protect\definitionname}
  \theoremstyle{remark}
  \newtheorem{rem}[thm]{\protect\remarkname}
  \theoremstyle{plain}
  \newtheorem{lem}[thm]{\protect\lemmaname}
  \theoremstyle{plain}
  \newtheorem{cor}[thm]{\protect\corollaryname}
  \theoremstyle{plain}
  \newtheorem*{thm*}{\protect\theoremname}
  \theoremstyle{plain}
  \newtheorem{prop}[thm]{\protect\propositionname}

\makeatother

\usepackage{babel}
  \providecommand{\corollaryname}{Corollary}
  \providecommand{\definitionname}{Definition}
  \providecommand{\lemmaname}{Lemma}
  \providecommand{\propositionname}{Proposition}
  \providecommand{\remarkname}{Remark}
  \providecommand{\theoremname}{Theorem}
\providecommand{\theoremname}{Theorem}

\begin{document}

\title[Parabolic Equations of Second Order with Critical Drift]{Non-divergence Parabolic Equations of Second Order with Critical Drift in Lebesgue Spaces}

\date{05/29/15}

\author{Gong Chen}


\email{gc@math.uchicago.edu}
\urladdr{http://www.math.uchicago.edu/\textasciitilde{}gc/}

\address{Department of Mathematics, The University of Chicago, 5734 South
University Avenue, Chicago, IL 60615, U.S.A}
\begin{abstract}
We consider uniformly parabolic equations and inequalities of second
order in the non-divergence form with drift
\[-u_{t}+Lu=-u_{t}+\sum_{ij}a_{ij}D_{ij}u+\sum b_{i}D_{i}u=0\,(\geq0,\,\leq0)\]
in some domain $Q\subset \mathbb{R}^{n+1}$. We prove
growth theorems and the interior Harnack inequality as the main results. In this paper, we will only assume
the drift $b$ is in certain Lebesgue
spaces which are critical under the parabolic scaling but not necessarily to be bounded. In the last
section, some applications of the interior Harnack inequality are presented.
In particular, we show there is a ``universal'' spectral gap for
the associated elliptic operator. The counterpart for uniformly elliptic
equations of second order in non-divergence form is shown in \cite{S10}.
\end{abstract}
\maketitle

\section{Introduction}

\subsection{General Introduction}

The qualitative properties of solutions to partial diff{}erential
equations have been intensively studied for a long time. In this paper,
we consider the qualitative properties of solutions to the uniformly
parabolic equation in non-divergence form,
\begin{equation}
-u_{t}+Lu:=-u_{t}+\sum_{ij}a_{ij}D_{ij}u+\sum_{i}b_{i}D_{i}u=0\label{eq:01}
\end{equation}
and the associated inequalities: $-u_{t}+Lu\geq 0$ and $-u_{t}+Lu\leq 0$. Throughout the paper, we use the notations $D_{i}:=\frac{\partial}{\partial x_{i}},\, D_{ij}:=\frac{\partial^{2}}{\partial x_{i}\partial x_{j}}$
and $u_{t}:=\frac{\partial u}{\partial t}$. We assume $b=\left(b_{1},\ldots,b_{n}\right)$ and $a_{ij}$'s are real  measurable,  $a_{ij}$'s also satisfy the {\em uniform parabolicity condition}
\begin{equation}
\forall\xi\in\mathbb{R}^{n},\,\nu^{-1}\left|\xi\right|^{2}\leq\sum_{i,j=1}^{n}a_{ij}(X)\xi_{i}\xi_{j},\,\,\,\,\,\,\,\sum_{i,j=1}^{n}a^{2}_{ij}\leq \nu^{2{}}\label{eq:02}
\end{equation}
with some constant $\nu\geq1$, 
 $\forall X=(x,t)$ in the domain of definition $Q\subset \mathbb{R}^{n+1}$. 

For the drift $b$, we will only require it is in certain
Lebesgue spaces which are critical under the parabolic scaling. To formulate our setting more precisely, we assume over the domain of definition $Q$,
\begin{equation}
\left\Vert b\right\Vert _{L_{x}^{p}L_{t}^{q}}:=\left(\intop\left[\int\left|b(x,t)\right|^{q}dt\right]^{\frac{p}{q}}dx\right)^{\frac{1}{p}}=:S(Q)<\infty,\label{eq:03}
\end{equation}
for some constants $p,\,q\geq1$ such that
\begin{equation}
\frac{n}{p}+\frac{2}{q}=1.\label{eq:04}
\end{equation}
By "critical", we mean that with the $L_{x}^{p}L_{t}^{q}$
norm, the drift is scaling invariant under the parabolic scaling: for
$r>0$,
 \[
x\rightarrow r^{-1}x,\,\, t\rightarrow r^{-2}t.
\]
Indeed, suppose $u$ satisfies
\[
-u_{t}+\sum_{ij}a_{ij}D_{ij}u+\sum_{i}b_{i}D_{i}u=0.
\] in a domain $Q\in \mathbb{R}^{n+1}$. Then for any constant $r>0$, let 
\[\tilde{x}=r^{-1}x,\,\tilde{t}=r^{-2}t.\] 
Then  $\widetilde{u}\left(\tilde{x},\tilde{t}\right)=u\left(r\tilde{x},r^{2}\tilde{t}\right)$
satisfies the equation
\[
-\widetilde{u}_{\tilde{t}}+\sum_{ij}\widetilde{a_{ij}}D_{ij}\widetilde{u}+\sum_{i}\tilde{b_{i}}D_{i}\widetilde{u}=0,
\]
in $Q_{r}:=\{(x,t),\,(rx,r^{2}t)\in Q \}$.
Note that $\tilde{b}=rb$, so
\[
S(Q_{r})=\left\Vert \tilde{b}\right\Vert _{L_{x}^{p}L_{t}^{q}}=\left\Vert b\right\Vert _{L_{x}^{p}L_{t}^{q}}=S(Q).
\]
In general, regarding the
scaling, intuitively, there is a competition between the transport term
and the diffusion part. One might expect that for the supercritical scaling
case,  $\frac{n}{p}+\frac{2}{q}>1$: the solutions of the equations
have discontinuities \cite{SVZ,GC}. For the critical  situation we
are considering here, we have H\"older continuous solutions, see Theorem \ref{hoelder}. Finally,
if the drift is subcritical with respect to the scaling, i.e. $\frac{n}{p}+\frac{2}{q}<1$,
we expect the solutions will be smooth. 

We will concentrate on the interior Harnack inequality for parabolic equations in non-divergence form with critical drift. Given constants $p$, $q$ satisfying condition \eqref{eq:04}, let $Q$ be an open set in $\mathbb{R}^{n+1} $, we define
\[ W_{p,q}(Q):= C(Q)\cap W_{p,q}^{2,1}(Q)\]
where $f\in W_{p,q}^{2,1}(Q)$ means $f_{t},\, D_{i}f,\, D_{ij}f\in\left(L_{x}^{p}L_{t}^{q}\right)_{loc}$.

With the assumptions above, the main result in this paper is then expressed by Theorem \ref{thm:Harnack}.
\begin{thm}[Interior Harnack Inequality]\label{thm:Harnack}
	Given constants $p$, $q$ satisfying condition \eqref{eq:04}, suppose $u\in W_{p,q}$ and  $-u_{t}+Lu=0$ in $Q_{2r}(Y):=B_{2r}(y)\times(s-(2r)^{2},s)$,
	$Y=(y,s)\in\mathbb{R}^{n+1}$ and $r>0$. If $u\geq0$, then
	\begin{equation}
	\sup_{Q^{0}}u\leq N\inf_{Q_{r}}u,\label{eq:harc}
	\end{equation}
	where $N:=N(n,\nu, p, q, S)$, $Q_{r}(Y):=B_{r}(y)\times(s-r^{2},s)$, $Q^{0}:=B_{r}(y)\times(s-3r^{2},s-2r^{2})$
	and $S$ is from condition \eqref{eq:03}.\end{thm}
\begin{rem*}
	We will see the most general form of Harnack principle in the later
	section on the applications of Harnack inequality.
\end{rem*}

Harnack inequalities have many important applications, not only in
differential equations, but also in other areas, such as diffusion
processes, geometry, etc. Unlike the classical maximum principle,
the interior Harnack inequality is far from obvious. For elliptic
and parabolic equations with measurable coefficients in the divergence
form, it was proved by Moser in the papers \cite{M61},\cite{M64}.
However, a similar result for non-divergence equations was obtained
15 years later after Moser's papers by Krylov and Safonov \cite{KS},
\cite{S80} in 1978-80. Their proofs relied on some improved versions
of growth theorems from the book by Landis \cite{EML}. These growth
theorems control the behavior of (sub-, super-) solutions of second
order elliptic and parabolic equations in terms of the Lebesgue measure
of areas in which solutions are positive or negative. In \cite{FS}, Ferretti and Safonov used growth theorems as a common
background for both divergence and non-divergence equations and used
these three growth theorems to derive the interior Harnack inequality.
Even in the one-dimensional case, the Harnack inequality fails for
equations of a ``joint'' structure,
which combine both divergence and non-divergence parts. One can find
detailed discussion in \cite{CS13}.

 At the beginning, the interior Harnack inequality was proved with bounded drift. Later
on, this condition was relaxed to subcritical drift $b$. For the subcritical case, we can always rescale the
problem. In small scale, the drift will work like a perturbation from the case without drift. But for the critical situation, our common tricks do not work. One can
find a historical overview of this progress in \cite{NU}.
For non-divergence elliptic equations of second order, in \cite{S10},
Safonov shown the interior Harnack inequality for the scaling critical
case $b\in L^{n}$. In this paper, we adapt Safonov's idea to the
parabolic setting. We will consider the case that the drift $b$ is in critical scaling Lebesgue spaces given by the conditions \eqref{eq:03} and \eqref{eq:04} above. In a later paper \cite{GC}, we will consider critical scaling Morrey
spaces with different approaches. Similar results for both divergence
form elliptic and parabolic equations are presented in \cite{NU}.

We will follow the unified approach to growth theorems and the interior
Harnack inequality developed in \cite{FS}. For this purpose, we need to prove three growth
theorems and derive the interior Harnack inequality as a consequence
for parabolic equations with critical drift formulated as above. We only present
the case $b\in L_{x}^{n}L_{t}^{\infty}$. For other cases ($p>n,\,q<\infty$) the proofs
are more or less identical to the situation we are considering here.
We will see remarks about them later on. For certain points, other
cases ($p>n,\,q<\infty$) are simpler than the endpoint case ($p=n,\,q=\infty$) we are discussing here. Namely, throughout the paper,
we assume over the domain of definition $Q$,

\begin{equation}
\intop\left[\sup_{t}\left|b(x,t)\right|\right]^{n}dx=S<\infty\label{eq:drift}
\end{equation}
where $\sup$ means essential supremum. One point compared with the
results for divergence form equations in \cite{NU} is that  our conclusions
for $L_{x}^{n}L_{t}^{\infty}$ case do not depend on the modulus of continuity
of the norm. For the sake of simplicity, we assume that all functions
(coefficients and solutions) are smooth enough. It is easy to get
rid of extra smoothness assumptions by means of standard approximation
procedures, see Section 7. We should notice that it is important to have appropriate
estimates for solutions with constants depending only on the prescribed
quantities, such as the dimension $n$, the parabolicity constant, etc.,
but not depending on ``additional'' smoothness.

This paper is organized as follows: In Section 1, we introduce our
basic assumptions and notations. In Section 2, we formulate a weak
version of the classical maximum principle, the Alexandrov-Bakelman-Pucci-Krylov-Tso
estimate, and some consequences of it. In Sections 3, 4, and
5, we formulate and prove three growth theorems. In Section 6, we derive
the interior Harnack inequality. In Section 7, we use approximation
to show all results are valid without extra smoothness assumption. Finally,
in Section 8, we discuss some applications of the interior Harnack
inequality. In particular, we show there is a ``universal'' spectral
gap between the principal eigenvalue and other eigenvalues for the elliptic
operator $L$ with drift $b\in L^{n}$ . 
In the appendix, we will prove the Alexandrov-Bakelman-Pucci-Krylov-Tso 
estimate we use in this paper.

\subsection{Notations:}

In this paper, we use summation convention.

``$A:=B$'' or ``$B=:A$'' is the definition of $A$ by means of
the expression $B$.
\begin{defn}\label{space}
For any open set $Q\subset\mathbb{R}^{n+1}$, we define the space
\begin{equation}
W(Q):= C(Q)\cap W_{n,\infty}^{2,1}(Q),\label{eq:W}
\end{equation} 
where $f\in W_{n,\infty}^{2,1}(Q)$ means $f_{t},\, D_{i}f,\, D_{ij}f\in\left(L_{x}^{n}L_{t}^{\infty}\right)_{loc}$. 
\end{defn}
$\mathbb{R}^{n}$ is the n-dimensional Euclidean space, $n\geq1$,
with points $x=(x_{1},\ldots,x_{n})^{t}$, where $x_{i}$'s are real
numbers. Here the symbol $t$ stands for the transposition of vectors
which indicates that vectors in $\mathbb{R}^{n}$ are treated as column
vectors. For $x=(x_{1},\ldots,x_{n})^{t}$ and $y=(y_{1},\ldots,y_{n})^{t}$
in $\mathbb{R}^{n}$, the scalar product $(x,y):=\Sigma x_{i}y_{i}$,
the length of $x$ is $|x|:=(x,x)^{\frac{1}{2}}$.

For a Borel set $\Gamma\subset\mathbb{R}^{n}$, $\bar{\Gamma}:=\Gamma\cup\partial\Gamma$
is the closure of $\Gamma$, $|\Gamma|$ is the n-dimensional Lebesgue
measure of $\Gamma$. Sometimes we use the same notation for the surface
measure of a subset $\Gamma$ of a smooth surface $S$.

For real numbers $c$, we denote $c_{+}:=\max(c,0)$, $c_{-}:=\max(-c,0)$.

In order to formulate our results, we need some standard definitions
and notations for the setting of parabolic equations.
\begin{defn}
Let $Q$ be an open connected set in $\mathbb{R}^{n+1}$, $n\geq1$.
The parabolic boundary $\partial_{p}Q$ of $Q$ is the set of all
points $X_{0}=(x_{0},t_{0})\in\partial Q$, such that there exists
a continuous function $x=x(t)$ on the interval $[t_{0},t_{0}+\delta)$
with values in $\mathbb{R}^{n}$, such that $x(t_{0})=x_{0}$ and $(x(t),t)\in Q$
for all $t\in(t_{0},t_{0}+\delta)$. Here $x=x(t)$ and $\delta>0$
depend on $X_{0}$. In particular, for cylinders $Q_{\Omega}=\Omega\times(0,T)$ with $\Omega\subset \mathbb{R}^n$,
the parabolic boundary $\partial_{p}Q_{\Omega}:=\left(\partial_{x}Q_{\Omega}\right)\cup\left(\partial_{t}Q_{\Omega}\right)$,
where $\partial_{x}Q_{\Omega}:=(\partial\Omega)\times(0,T)$, $\partial_{t}Q_{\Omega}:=\Omega\times\{0\}$.
\end{defn}
We will use the following notation for the "standard" parabolic cylinder. For $Y=(y,s)$ and $r>0$, we define $Q_{r}(Y):=B_{r}(y)\times(s-r^{2},s),$
where $B_{r}(y):=\{x\in\mathbb{R}^{n}:|x-y|<r\}$.

\section{Preliminaries}

In this section, we briefly discuss some well-known theorems and results
which are crucial for us to carry out the discussion in the later
parts of this paper.
We use the notation $u\in W(Q)= C(Q)\cap W_{n,\infty}^{2,1}(Q)$ in the sense of Definition \ref{space}. Also we denote $S=\intop\left[\sup_{t}\left|b(x,t)\right|\right]^{n}dx<\infty$
\begin{thm}[Alexandrov-Bakelman-Pucci-Krylov-Tso
	estimate]
\label{thm:(Alexandrov-Bakelman-Pucci-Krylo}\label{thm:ABPKT} Suppose $u\in W(\Omega)$ and $\Omega\subset Q_{r}$ and $-u_{t}+Lu\geq f$. If
$\sup_{\partial_{p}\Omega}u\leq0$, then
\begin{equation}
\sup_{\Omega}u\leq N(\nu,n,S)r\left\Vert f\right\Vert _{L_{x}^{n}L_{t}^{\infty}}.\label{eq:AMPKT}
\end{equation}

\end{thm}
We will present the detailed proof of above theorem in Appendix A.
\begin{rem}
In \cite{AIN}, Nazarov shown the Alexandrov-Bakelman-Pucci-Krylov-Tso
estimate holds for the drift $b\in L_{x}^{p}L_{t}^{q}$, i.e.,
\[
\left\Vert b\right\Vert _{L_{x}^{p}L_{t}^{q}}=\left(\intop\left[\int\left|b(x,t)\right|^{q}dt\right]^{\frac{p}{q}}dx\right)^{\frac{1}{p}}<\infty,
\]
for
\[
\frac{n}{p}+\frac{1}{q}\leq1,\,\,p,q\geq1.
\]
The proof was based on Krylov\textquoteright{}s ideas and methods \cite{NVK}.
In this paper, we mainly focus the case $p=n$ and $q=\infty$. The
approach in this paper is easily modified to show the general
scaling invariant cases, i.e., for some constants $p,\,q\geq1$ such that 
\[
\frac{n}{p}+\frac{2}{q}=1,\,\, q<\infty.
\]
We will discuss this point again later on.\end{rem}
\begin{thm}[Maximal Principle]\label{thm:MP}
 Let $Q$ be a bounded open set
in $\mathbb{R}^{n+1}$, and let a function $u\in C^{2,1}\left(\bar{Q}\backslash\partial_{p}Q\right)\cap C(\bar{Q})$
satisfy the inequality $-u_{t}+Lu\geq0$ in $Q$. Then
\begin{equation}
\sup_{Q}u=\sup_{\partial_{p}Q}u\label{eq:-1}
\end{equation}

\end{thm}
As an easy consequence of the maximal principle and the Alexandrov-Bakelman-Pucci-Krylov-Tso
estimate, we have the well-known comparison principle.
\begin{thm}[Comparison Principle]
\label{thm:Comparison}Let $Q$ be a bounded
domain in $\mathbb{R}^{n+1}$, $u,\, v\in W(Q)\cap C(\bar{Q})$,
$-u_{t}+Lu\leq-v_{t}+Lv$ in $Q$, and $u\geq v$ on $\partial_{p}Q$,
then $u\geq v$ on $\bar{Q}$.
\end{thm}

\section{First Growth Theorem}

Suppose $R$ is the region in a cylinder where a subsolution $u$ of our equation is positive. 
The first growth theorem, Theorem \ref{thm:(First-Growth-Theorem)}, basically tells us if the measure of $R$ is small, then
the maximal value of $u$ over half of the cylinder is strictly less than the
maximal value over the whole cylinder. In other words, it gives us
some  quantitative decay properties.

Before we start to prove the first growth theorem, we need to prove
several intermediate results based on the comparison principle and the
Alexandrov-Bakelman-Pucci-Krylov-Tso estimate. Let us first do some
preliminary calculations in order to carry out some comparison arguments.

For fixed numbers $\alpha>0$ and $0<\epsilon<1$, in the cylinder
$Q=B_{r}(0)\times(-r^{2},(\alpha-1)r^{2})$, we can define
\begin{equation}
\psi_{0}=\frac{(1-\epsilon^{2})(t+r^{2})}{\alpha}+\epsilon^{2}r^{2}\label{eq:11}
\end{equation}
and
\begin{equation}
\psi_{1}=\left(\psi_{0}-|x|^{2}\right)_{+}\label{eq:12}
\end{equation}
where $\left(\cdot\right)_{+}$ means positive part of the function.
And we also define
\begin{equation}
\psi=\psi_{1}^{2}\psi_{0}^{-q}\label{eq:13}
\end{equation}
for some number $q\geq2$ to be determined later. First of all, we
notice $\psi$ is $C^{2,1}$ in $\widetilde{Q}:=\left\{ (x,t)|\,|x|^{2}<\psi_{0},\,-r^{2}<t<(\alpha-1)r^{2}\right\} $.
It is clear that $-\psi_{t}+L\psi=0$ if $\psi_{0}\leq|x|^{2}$. Now
if $\psi_{0}>|x|^{2}$, by some computations, we obtain
\[
-\psi_{t}+\sum a_{ij}D_{i,j}\psi=\psi_{0}^{-q}\left[8a_{ij}x_{i}x_{j}-4\psi_{1}trace(a_{ij})+\frac{(1-\epsilon^{2})q}{\alpha\psi_{0}}\psi_{1}^{2}-2\frac{(1-\epsilon^{2})}{\alpha}\psi_{1}\right].
\]
Set $F_{1}=\frac{2}{\alpha}+8n\nu^{-1}$, and $\xi=\frac{\psi_{1}}{\psi_{0}}$
then

\begin{equation}
-\psi_{t}+\sum a_{ij}D_{ij}\psi\geq\psi_{0}^{1-q}\left[\frac{(1-\epsilon^{2})q}{\alpha}\xi^{2}-F_{1}\xi+8\lambda\right].\label{eq:13.1}
\end{equation}
Pick
\begin{equation}
q=2+\frac{\alpha}{32(1-\epsilon^{2})},\label{eq:14}
\end{equation}
so that the quadratic form in \eqref{eq:13.1} is non-negative. Then
we get
\[
-\psi_{t}+\sum a_{ij}D_{ij}\psi\geq0,\,\,\,\,\forall(x,t)\in Q.
\]
We also notice that
\[
\psi(x,-r^{2})\leq(\epsilon r)^{-2q+4},\,\,\,\,\forall|x|\leq r,
\]
and
\begin{equation}
\psi\left(x,(\alpha-1)r^{2}\right)\geq\frac{9}{16}r^{-2q+4},\,\,\,\,\,\forall|x|\leq\frac{r}{2}.\label{eq:15}
\end{equation}
Finally, we notice that by the monotonicity of $\psi$ with respect
to $t\in\left[-r^{2},(\alpha-1)r^{2}\right]$ for $x=0$, we obtain
\begin{equation}
\psi\left(0,t\right)\geq\frac{9}{16}r^{-2q+4},\,\, t\in\left[-r^{2},(\alpha-1)r^{2}\right].\label{eq:16}
\end{equation}
With the help of $\psi$ we just constructed in \eqref{eq:13},in Lemma \ref{lem:small1}, we first
show that when the drift $b$ is small enough,  if we stay away from
the lateral boundary of a standard cylinder. Then we have a lower bound
for a positive supersolution $u$ ($-u_{t}+Lu\leq0$)
in a spacial region of the bottom of the cylinder, then we
also have a lower bound for $u$ in the same region on the top of
the cylinder. Intuitively, it basically tells us $u$ will not decay
dramatically in the same spatial region if we stay away from the lateral
boundary.
\begin{lem}
\label{lem:small1}Let $\alpha$ be positive constant and $-u_{t}+Lu\leq0$
in $\Omega$. Suppose $Q:=B_{r}(0)\times(-r^{2},(\alpha-1)r^{2})\subset\Omega$
and $u>0$ in $Q$. Then there are positive constants $s_{1}:=s_{1}(n,\nu)$,
$C_{1}:=C_{1}(n,\nu)$ and $k:=k(n,\nu,\alpha)$ such that if
\begin{equation}
u\geq\ell\label{eq:17}
\end{equation}
on $B_{\frac{r}{2}}(0)\times\{-r^{2}\}$ and $\left(\left\Vert b\right\Vert _{L_{x}^{n}L_{t}^{\infty}(Q)}\right)^{n}\leq s_{1}\leq S$,
then
\begin{equation}
u\geq C_{1}(\frac{1}{2})^{k}\ell\label{eq:18}
\end{equation}
on $B_{\frac{r}{2}}(0)\times\{(\alpha-1)r^{2}\}$.
\end{lem}
The above lemma is a special situation of Lemma 7.39 in \cite{GL2}
with some mollification.
\begin{proof}
We apply the results from the preliminary calculations with case $\epsilon=\frac{1}{2}$. Consider
\begin{equation}
v=u-\ell(\frac{1}{2}r)^{2q-4}\psi.\label{eq:19}
\end{equation}
Notice that with the $q$ from the above calculation, we have
\begin{equation}
-\psi_{t}+L\psi\geq b_{i}D_{i}\psi=-4\psi_{1}\psi_{0}^{-q}(b,x)\geq-4|b|r(\frac{1}{2})^{-2q}r^{2-2q}.\label{eq:110}
\end{equation}
Also it is clear $\psi=0$ for $|x|=r$ . So we can conclude that
$v\geq0$ on $\partial_{p}\widetilde{Q}$ by the above calculation
and $u\geq\ell$ on the bottom. Finally, we apply Theorem \ref{thm:(Alexandrov-Bakelman-Pucci-Krylo}
to $-v$, we obtain
\begin{equation}
v\geq-N(n,\nu,s_{1})\ell(\frac{1}{2})^{-4}s_{1}^{\frac{1}{n}}\label{eq:111}
\end{equation}
 in $\widetilde{Q}$. In other words, we have
\begin{equation}
u\geq\ell(\frac{1}{2}r)^{2q-4}\psi-N(n,\nu,s_{1})\ell(\frac{1}{2})^{-4}s_{1}^{\frac{1}{n}}.\label{eq:112}
\end{equation}
By the calculation above again, we conclude that
\begin{equation}
u\geq\ell(\frac{1}{2})^{-4}\left[\frac{9}{16}(\frac{1}{2})^{2q}-N(n,\nu,s_{1})s_{1}^{\frac{1}{n}}\right].\label{eq:113}
\end{equation}
for $|x|\leq\frac{r}{2}$. Pick $s_{1}:=s_{1}(n,\nu)$ small enough ($N(n,\nu,s_{1})$
is decreasing when $s_{1}$ decays) to force the inequality , 
\begin{equation}
\left[\frac{9}{16}(\frac{1}{2})^{2q}-N(n,\nu,s_{1})s_{1}^{\frac{1}{n}}\right]\geq\frac{1}{2}(\frac{1}{2})^{2q}.\label{eq:114}
\end{equation}
 We  conclude
\begin{equation}
u\geq C_{1}(\frac{1}{2})^{k}\ell\label{eq:115}
\end{equation}
on $B_{\frac{r}{2}}(0)\times\{(\alpha-1)r^{2}\}$, with $k=2q-4$
and $C_{1}$ does not depend on $u$. As a byproduct, we can also
conclude that
\begin{equation}
u(0,t)\geq C_{1}(\frac{1}{2})^{k}\ell,\,\,\,\forall t\in[-r^{2},(\alpha-1)r^{2}].\label{eq:116}
\end{equation}

\end{proof}
Next, by iterating Lemma \ref{lem:small1} and applying the pigeonhole
principle, we show that under the same assumptions on $u$ as above
but without the assumption of the smallness of $b$, if $u$ has a
lower bound on the bottom of a cylinder, then $u$ still has a lower
bound for later time at least in some small region in space.
\begin{lem}
\label{lem:stol}Let $\alpha$ a be positive constant and $-u_{t}+Lu\leq0$
in $\Omega$. Suppose $Q:=B_{r}(0)\times(-r^{2},(\alpha-1)r^{2})\subset\Omega$
and $u>0$ in $Q$. If $u$ has a lower bound on the bottom of the
cylinder,
\begin{equation}
u\geq\ell\label{eq:117}
\end{equation}
\textup{on $B_{r}(0)\times\{-r^{2}\}$}, then it has a lower bound
on $B_{\frac{r}{2}}(0)\times\{(\alpha-1)r^{2}\}$,
\begin{equation}
u\geq C(n,\nu,S,\alpha)\ell\label{eq:118}
\end{equation}
for some positive constant $C:=C(n,\nu,S,\alpha)$ where $S:=\left(\left\Vert b\right\Vert _{L_{x}^{n}L_{t}^{\infty}}\right)^{n}$.
In particular
\begin{equation}
u(0,(\alpha-1)r^{2})\geq C(n,\nu,S,\alpha)\ell\label{eq:119}
\end{equation}
\end{lem}
\begin{proof}
The result can be proved by iterating Lemma \ref{lem:small1}. We
divide
\[
\left[B_{r}(0)\times(-r^{2},(\alpha-1)r^{2})\right]\backslash\left[B_{\frac{r}{2}}(0)\times(-r^{2},(\alpha-1)r^{2})\right]
\]
into $m$ pieces of cylindrical shells where $m=[\frac{S}{s_{1}}]+1$,
where $s_{1}$ satisfies the conditions in Lemma \ref{lem:small1}
. If we denote $r_{k}=\frac{r}{2}+\frac{1}{2}\frac{kr}{m}$ for $k=1,\ldots m$,
each of those shells is of the form
\[
V_{k}=\left(B_{r_{k}}(0)\backslash B_{r_{k-1}}(0)\right)\times(-r^{2},(\alpha-1)r^{2})
\]
where $k=1,\ldots,m$. Then at least over one of these shells, say,
$V_{k_{0}}$, such that $\left[\left\Vert b\right\Vert _{L_{x}^{n}L_{t}^{\infty}(V_{k_{0}})}\right]^{n}\leq s_{1}$,
i.e. the norm of the drift is small over $V_{k_{0}}$. For any $\left|y_{k_{0}}\right|=\frac{r}{2}+\frac{1}{2}\frac{k_{0}r}{m}-\frac{1}{4}\frac{r}{m}$,
we can apply the above lemma iteratively to the cylinders

\begin{eqnarray*}
B_{\frac{1}{4}\frac{r}{m}}\left(y_{k_{0}}\right)\times\left(-r^{2},-r^{2}+\alpha\left(\frac{1}{4}\frac{r}{m}\right)^{2}\right),\ldots,\qquad\qquad\qquad\qquad\qquad\qquad\qquad\qquad\qquad\\
\qquad\qquad\qquad B_{\frac{1}{4}\frac{r}{m}}\left(y_{k_{0}}\right)\times\left(-r^{2}+h\alpha\left(\frac{1}{4}\frac{r}{m}\right)^{2},-r^{2}+(h+1)\alpha\left(\frac{1}{4}\frac{r}{m}\right)^{2}\right)
\end{eqnarray*}
where $h\in\mathbb{N}$ such that $-r^{2}+(h+1)\left(\frac{1}{4}\frac{r}{m}\right)^{2}=(\alpha-1)r^{2}$.
In other words, we apply Lemma \ref{lem:small1} iteratively and put
those small cylinders together vertically in one particular cylinder
shell. We have $\forall\left|y_{k_{0}}\right|=\frac{r}{2}+\frac{1}{2}\frac{k_{0}r}{m}-\frac{1}{4}\frac{r}{m}$
and $\forall t\in\left[-r^{2},(\alpha-1)r^{2}\right]$, $u(y_{k_{0}},t)\geq C'(n,\nu,S,\alpha)\ell$
for some constant $C'$. Finally, with the maximal principle applied
to $-u$ and the lower bound on $u$ on the bottom, we conclude that on $B_{\frac{r}{2}}(0)\times\{(\alpha-1)r^{2}\}$,
for some constant $C(n,\nu,S,\alpha)$
\begin{equation}
u\geq C(n,\nu,S,\alpha)\ell,\label{eq:120}
\end{equation}
and
\begin{equation}
u(0,(\alpha-1)r^{2})\geq C(n,\nu,S,\alpha)\ell.\label{eq:121}
\end{equation}

\end{proof}
Analogous results to Lemmas \ref{lem:small1}, \ref{lem:stol}
also hold for a slanted cylinder setting. For a fixed point $Y=(y,s)\in\mathbb{R}^{n+1}$
with $s>0$, and $r>0$, define the slanted cylinder

\begin{equation}
V_{r}=V_{r}(Y):=\left\{ X=(x,t)\in\mathbb{R}^{n+1};\,\left|x-\frac{t}{s}y\right|<r,0<t<s\right\} .\label{eq:122}
\end{equation}

\begin{lem}
\label{lem:slant1}Let a function $u\in C^{2,1}(\overline{V_{r}})$
satisfy $-u_{t}+Lu\leq0$ in a slanted cylinder $V_{r}$, which is defined
in \eqref{eq:122} with $Y=(y,s)\in\mathbb{R}^{n+1}$, $s>0$, $r>0$,
such that

\begin{equation}
K^{-1}r|y|\leq s\leq Kr^{2}\label{eq:123}
\end{equation}
where $K>0$ is a constant. In addition, suppose \textup{$u\geq\ell$}
on $D_{r}:=B_{r}(0)\times\{0\}$. Then

\begin{equation}
u\geq C(n,\nu,S,K)\ell\label{eq:124}
\end{equation}
on $B_{\frac{1}{2}r}(y)\times\{s\}$. In particular,
\begin{equation}
u(y,s)\geq C(n,\nu,S,K)\ell.\label{eq:125}
\end{equation}
\end{lem}
\begin{proof}
This can be proved using the result for the standard cylinder setting
with a change of variables. First of all, we notice that $s/r^{2}\leq K$
so $\alpha$ in above lemmas is bounded by $K$. We rescale $V_{r}(Y)$
to $V_{1}(Y)$ since all the quantities we are considering are scaling
invariant. Next we make a change of variables. We notice that with
$k_{i}:=\frac{y_{i}}{s}$ and $\frac{|y_{i}|}{s}=|k_{i}|\leq\frac{|y|}{s}\leq K$,
then define $w_{i}=x_{i}-k_{i}t$ and $z=t$. In this coordinate,
the slanted cylinder is transformed to a standard cylinder. The equation
with respect to the new coordinate is
\begin{equation}
-u_{z}+\sum_{ij}a_{ij}D_{w_{i}w_{j}}u+\sum_{i}(b_{i}+k_{i})D_{w_{i}}u\leq0.\label{eq:126}
\end{equation}
Then we apply the stand cylinder result to the equation with respect
to coordinate $(w,z)$. we have
\[
u(y,s)\geq C(n,\nu,S,K)\ell.
\]

\end{proof}
Now the useful slanted cylinder lemma \cite{FS} follows easily from
Lemma \ref{lem:slant1}. We can apply Lemma \ref{lem:slant1} to $1-u$
after we multiply $u$ by a constant to reduce our problem to the
case $1=\sup_{V_{r}(Y)}u_{+}$. We have the following result:
\begin{lem}[Slanted Cylinder Lemma]
\label{lem:(Slant-Cylinder-Lemma)} Let a function
$u\in C^{2,1}(\overline{V_{r}})$ satisfy $-u_{t}+Lu\geq0$ in a slanted
cylinder $V_{r}$, which is defined in \eqref{eq:122} with $Y=(y,s)\in\mathbb{R}^{n+1}$,
$s>0$, $r>0$, such that

\begin{equation}
K^{-1}r|y|\leq s\leq Kr^{2}\label{eq:127}
\end{equation}
where $K>1$ is a constant. In addition, suppose $u\leq0$ on $D_{r}:=B_{r}(0)\times\{0\}$.
Then

\begin{equation}
u(Y)\leq\beta_{2}\sup_{V_{r}(Y)}u_{+}\label{eq:128}
\end{equation}
with a constant $\beta_{2}=\beta_{2}(\nu,n,K,S)<1$.
\end{lem}
With the above comparison results, we can proceed to our proof of the first growth theorem. We will do the construction in the spirit of \cite{S10}
and use the structure of the parabolic maximal principle.
\begin{thm}[First Growth Theorem]
\label{thm:(First-Growth-Theorem)} Let a function
$u\in C^{2,1}(\overline{Q_{r}})$ where $r>0$ and $Q_{r}=Q_{r}(Y)$
is a standard cylinder, in $\mathbb{R}^{n+1}$ containing $Y:=(y,s)$.
Suppose $-u_{t}+Lu\geq0$ in $Q_{r}$, then $\forall\beta_{1}\in(0,1)$,
there exists $0<\mu<1$ such that if we know
\begin{equation}
\left|\left\{ u>0\right\} \cap Q_{r}(Y)\right|\leq\mu|Q_{r}(Y)|,\label{eq:129}
\end{equation}
then

\begin{equation}
\mathcal{M}_{\frac{r}{2}}(Y)\leq\beta_{1}\mathcal{M}_{r}(Y),\label{eq:130}
\end{equation}
where
\[
\mathcal{M}_{r}(Y):=\max_{Q_{r}(Y)}u_{+}.
\]
In addition, we also notice that $\beta_{1}\rightarrow0^{+}$ as $\mu\rightarrow0^{+}$.
\end{thm}
Roughly, in order to prove the first growth theorem, one can use a
elliptic type argument similar to the one in \cite{S10} to find
a certain region where the drift is small. Then we just treat the
small drift as a perturbation or an error term in the proof of the
case without drift term. With the above comparison results. We can use a slanted cylinder to
joint an arbitrary point in the standard cylinder and some portion
of the region we found by the elliptic argument. Finally we apply Lemma \ref{lem:slant1} to the slanted cylinder to
have some control of the value of $u$.
\begin{rem}
\label{point}First of all, we make some reductions. In our problem,
we want to show under some conditions, given $-u_{t}+Lu\geq0$ in
a cylinder $Q_{r}(Y)$, and some information about the set $\{u\leq0\}$,
we want to show that
\[
\mathcal{M}_{\frac{r}{2}}(Y)\leq\beta_{1}\mathcal{M}_{r}(Y).
\]
Clearly, in order to derive the above estimate, we only need to consider
positive part of $u$. We observe that to obtain the above estimate,
it actually suffices to get
\begin{equation}
u(Y)\leq\beta_{1}\mathcal{M}_{r}(Y),\label{eq:extra}
\end{equation}
for some $\beta_{1}\in(0,1)$. Indeed, for an arbitrary point $Z\in Q_{\frac{r}{2}}(Y)$,
we notice $Q_{\frac{r}{2}}(Z)\subset Q_{r}(Y)$, we can apply the
above estimate \eqref{eq:extra} to $Q_{\frac{r}{2}}(Z)$ with $Y$
replaced by $Z$ and $r$ replaced by $\frac{r}{2}$ with some measure
condition $\mu'$. With respect to the measure condition in the
first growth theorem, we also observe that
\[
\left|\left\{ u>0\right\} \cap Q_{\frac{r}{2}}(Z)\right|\leq\left|\left\{ u>0\right\} \cap Q_{r}(Y)\right|\leq\mu\left|Q_{r}\right|=2^{n+2}\mu\left|Q_{\frac{r}{2}}(Z)\right|.
\]
So we just need to take $\mu=2^{-n-2}\mu'$ for the measure condition
in the first growth theorem.
\end{rem}

\begin{rem}
\label{less1}In the first growth theorem, we point out that when
$\mu$ goes to $0$, $\beta_{1}$ goes to $0$. But actually, it suffices
to show that if $\mu$ is small enough, $\beta_{1}:=\beta_{1}(n,\nu,S,\mu)<1$.
Indeed,  we can apply the above estimate
inductively to $Q_{2^{-k}}(Y)$. To illustrate the idea, without loss
of generality, we can apply the above estimate to $Z_{1}\in Q_{\frac{r}{2}}(Y)$
and $Q_{\frac{r}{2}}(Z_{1})$, we have $u(Z_{1})<\beta_{1}<1$, $\forall Z_{1}\in Q_{\frac{r}{2}}(Y)$. Then we apply the above estimate again to $u/\beta_{1}$ in all
points $Z_{2}\in Q_{\frac{r}{4}}(Y)$ with $Q_{\frac{r}{2^{3}}}(Z_{2})$ to obtain $\left(u/\beta_{1}\right)(Z_{2})<\beta_{1}$, i.e., $\forall Z_{2}\in Q_{\frac{r}{4}}(Y)$,
$u(Z_{2})<\beta_{1}^{2}$ provided $\mu^{(2)}\leq(2^{n+2})^{-2}\mu$
in the first growth lemma. We can do this process inductively. For
any $\beta_{0}\in(0,1)$, we can find $m\in\mathbb{N}$, such that
$\beta_{1}^{m}<\beta_{0}$, we choose $\mu^{(m)}\leq(2^{n+2})^{-m}\mu$.
After we do the above process $m$ times, we conclude that $\forall Z_{m}\in Q_{\frac{r}{2^{m}}}(Y)$,
$u(Z_{m})<\beta_{1}^{m}<\beta_{0}$. In particular, we infer that
\[
u(Y)\leq\beta_{0}\mathcal{M}_{r}(Y).
\]
\end{rem}
\begin{proof}
Since every quantity is scaling invariant, we can assume $r=1$. And
we can multiply $u$ by a constant, so without loss of generality, we
can also assume $\mathcal{M}_{1}(Y)=1$. Also we assume $u(Y)>0$,
otherwise the result is trivial.\\
\noun{Step 1}: We show the first growth theorem holds for $\left(\left\Vert b\right\Vert _{L_{x}^{n}L_{t}^{\infty}(Q)}\right)^{n}\leq s_{0}:=s_{0}(\beta_{1},n,\nu)<S$,
where $s_{0}$ is small enough. Consider
\begin{equation}
v(X)=v(x,t)=u(x,t)+t-s-|x-y|^{2}\label{eq:131}
\end{equation}
in $Q:=\left\{ v>0\right\} \cap Q_{1}(Y)$. Clearly, $Q\neq\emptyset$
since $v(Y)=u(Y)>0$ and $Y\in\partial Q_{1}(Y)$. It is easy to see
that $v\leq u$ in $Q$. By the measure condition, we obtain
\[
\left|Q\right|\leq\left|\left\{ u>0\right\} \cap Q_{1}(Y)\right|\leq\mu|Q_{1}(Y)|\leq\mu.
\]
Note that $v\leq0$ on $\partial_{p}Q_{1}(Y)$, so $v=0$ on $\partial_{p}Q$.
Since $-u_{t}+Lu\geq0$, we obtain
\begin{equation}
(-\partial_{t}+L)v\geq0-1-2trace(a_{ij})-2|b|\geq-1-2nv^{-1}-2|b|.\label{eq:132}
\end{equation}
By the Alexandrov-Bakelman-Pucci-Krylov-Tso estimate,
\begin{equation}
u(Y)\leq N(\nu,n,S)\left\Vert -1-2nv^{-1}-2|b|\right\Vert _{L_{x}^{n}L_{t}^{\infty}(Q)}\leq N_{1}(\nu,n,S)\left(\mu^{\frac{1}{n}}+s_{0}^{\frac{1}{n}}\right).\label{eq:133}
\end{equation}
Fix a $\beta_{1}\in(0,1)$, let $s_{1}^{\frac{1}{n}}:=\frac{\beta_{1}}{2N_{1}}$,
then let $s_{0}<s_{1}$ and $\mu<s_{1}^{n}$. Then we obtain
\begin{equation}
u(Y)<\beta_{1}.\label{eq:134}
\end{equation}
\noun{Step 2:} We follow the spirit in \cite{S10} to find a region,
such that the drift over it is small.

By the above discussion, we can choose $s_{1}$ from \noun{Step 1}.
Now just for convenience, we translate $Y$ to $(0,0)$. We divide
$\left[B_{1}(0)\times(-1,0)\right]\backslash\left[B_{\frac{1}{2}}(0)\times(-1,0)\right]$
into $m$ pieces of cylindrical shells where $m=[\frac{S}{s_{1}}]+1$.
If we denote $r_{k}=\frac{1}{2}+\frac{1}{2}\frac{k}{m}$ for $k=0,1,\ldots m$,
each of those shells is of the form
\[
V_{k}=\left(B_{r_{k}}(0)\backslash B_{r_{k-1}}(0)\right)\times(-1,0)
\]
 where $k=1,\ldots m$. Then at least on one of these shells, say,
$V_{k_{0}}$, such that $\left\Vert b\right\Vert _{L_{x}^{n}L_{t}^{\infty}(V_{k_{0}})}\leq s_{1}$,
i.e. the norm of the drift is small over $V_{k_{0}}$ is small. For
any $\left|y_{k_{0}}\right|\in\left[\frac{1}{2}+\frac{1}{2}\frac{k_{0}}{m}-\frac{3}{8}\frac{1}{m},\frac{1}{2}+\frac{1}{2}\frac{k_{0}}{m}-\frac{1}{8}\frac{1}{m}\right]$,
and $\forall t_{0}\in\left[-1+\frac{1}{(8m)^{2}},0\right]$, denote
$Y_{0}=(y_{k_{0}},t_{0})$. Then $Q_{\frac{1}{8m}}(Y_{0})\subset Q_{1}(Y)$.
Now we take $\mu'\leq\left(\frac{1}{8m}\right)^{n+2}\mu$ where $\mu$
is from the calculation in \noun{Step 1} to get
\begin{equation}
u(Y_{0})<\beta_{1}.\label{eq:135}
\end{equation}
\noun{Step 3}: Now we can apply the preliminary comparison results to $v=1-u$.
From \noun{step 2}, we can find a cylinder shell, in which $u\leq\beta_{1}$.
Then we can joint $Y$ and a $n$-dimensional ball in this region with
a slanted cylinder. More precisely, from \noun{step 2}, we can fix a
point $Y^{0}$ such that $\left|y^{0}\right|=\frac{1}{2}+\frac{1}{2}\frac{k_{0}}{m}-\frac{1}{4}\frac{1}{m}$
and a ball $B_{\frac{1}{8m}}(y^{0})$. Then we use a slanted cylinder
of radius $\frac{1}{8m}$ to joint $B_{\frac{1}{8m}}(y^{0})$ and $Y$.
Now apply Lemma \ref{lem:slant1} to $v=1-u$. We notice the quantity
$K$ in Lemma \ref{lem:slant1} in this situation is bounded above, independent of $u$, and the upper bound only depends on $m$.
Clearly,
$v\geq1-\beta_{1}=:\ell$ on the bottom, then Lemma \ref{lem:slant1}
gives
\[
v(Y)=1-u(Y)\geq C(n,v,S,K)(1-\beta_{1}).
\]
So we can conclude that
\begin{equation}
u(Y)\leq1-C(n,v,S,K)(1-\beta_{1})=:\mbox{\ensuremath{\beta}}<1.\label{eq:136}
\end{equation}
Finally, by remarks \ref{point} and \ref{less1}, we are done.\end{proof}
\begin{rem}
For other scaling invariant drift cases, i.e.,
\[
\left(\intop\left[\int\left|b(x,t)\right|^{q}dt\right]^{\frac{p}{q}}dx\right)^{\frac{1}{p}}<\infty
\]
 for some constants $p,\,q\geq1$ such that 
\[
\frac{n}{p}+\frac{2}{q}=1,\,\, q<\infty.
\]
The first growth theorem is easier to show. After rescaling and do
translation again, we may still assume $r=1$ and $Y=0$. Similarly as
above, we know when $\beta_{1}$ is fixed, the fist growth theorem holds when the norm of the drift $b$  is small norm, say, $\left\Vert b\right\Vert _{L_{x}^{p}L_{t}^{q}}<s_{1}$. Similarly as\noun{ step 2} above, we can find
cylinder shell over which the drift is small, and over this region
$u(Y^{0})<\beta_{1}$. Next we divide $Q_{1}(0)$ evenly along $t$
direction to $m^{2}$ pieces. i.e. each shell is of the form $U_{k}=B_{1}(0)\times[-\frac{k}{(8m)^{2}},-\frac{k-1}{(8m)^{2}}]$,
for $m$ large (say, here $m$ is larger than the number $m=[\frac{S}{s_{1}}]+1$), we can find at least over one of $U_{k}$'s, say, $U_{j_{0}}$
\[
\left\Vert b\right\Vert _{L_{x}^{p}L_{t}^{q}\left(U_{j_{0}}\right)}<s_{1}.
\]
Now for $(y_{j_{0}},t_{j_{0}}=\frac{j_{0}}{(8m)^{2}})$, we can build
a parabolic cylinder $Q_{\frac{1}{8m}}\left(y_{j_{0}},\frac{j_{0}}{(8m)^{2}}\right)$.
Over this cylinder, by \noun{step 1}, we know
\[
u(y_{j_{0}},t_{j_{0}})<\beta_{1}.
\]
The above estimate holds for $y_{j_{0}}\in B_{1-\frac{1}{8m}}(0)$.
Finally, we apply the maximal principle and \noun{step 2} above, we
have
\[
u(y,s)=u(Y)<\beta_{1}.
\]

\end{rem}

\section{Second Growth Theorem}

The slanted cylinder lemma, i.e., Lemma \ref{lem:(Slant-Cylinder-Lemma)}
above plays a crucial role in this section to build a connection between
different time slides. The second growth theorem helps us control the oscillation between different time slides. We
follow the arguments in \cite{FS}.
\begin{thm}[Second Growth Theorem]
\textup{\label{thm:(Second-Growth-Theorem).}
Let a function $u\in C^{2,1}\left(\overline{Q_{r}}\right)$, where
$Q_{r}:=Q_{r}(Y)$, $Y:=(y,s)\in\mathbb{R}^{n+1}$, $r>0$, and let
$-u_{t}+Lu\geq0$ in $Q_{r}$. In addition, suppose $u\leq0$ on $D_{\rho}:=B_{\rho}(z)\times\{\tau\}$,
where $B_{\rho}(z)\subset B_{r}(y)$ and
\begin{equation}
s-r^{2}\leq\tau\leq s-\frac{1}{4}r^{2}-\rho^{2}.\label{eq:21}
\end{equation}
Then
\begin{equation}
u(Y)\leq\beta_{3}\sup_{Q_{r}(Y)}u_{+}\label{eq:22}
\end{equation}
where $\beta_{3}:=\beta_{3}(n,\nu,\rho/r,S)<1$ is a constant.}\end{thm}
\begin{proof}
After rescaling and translation in $\mathbb{R}^{n+1}$, we reduce
our problem to $r=1$, and $(z,\tau)=(0,0)\in\mathbb{R}^{n+1}$. For
an arbitrary point $Y'\in Q_{\frac{1}{2}}(Y)$, we can apply the slanted
cylinder lemma to the slanted cylinder $V_{\rho}(Y')\subset Q_{1}(Y)$.
Note that in this situation, the constant $K$ in slanted cylinder lemma
only depends on $\rho$. Therefore, with the parameter $\beta_{2}$
from the slanted cylinder lemma, we have
\[
u(Y')\leq\beta_{2}\sup_{V_{\rho}(Y')}u_{+}\leq\beta_{2}\sup_{Q_{1}(Y)}u_{+}.
\]
The above estimate holds for all $Y'\in Q_{\frac{1}{2}}(Y)$. Then,
in particular, we obtain
\[
\sup_{Q_{\frac{1}{2}}(Y)}u_{+}\leq\beta_{2}\sup_{Q_{1}(Y)}u_{+}.
\]

\end{proof}
Now we establish an estimate similar to above with  more explicit
dependence of the constant on the ratio $\rho/r$.
\begin{lem}
\label{lem:quotient}Let a function $v\in C^{2,1}\left(\overline{Q_{r}}\right)$
satisfy $v\geq0$, $-v_{t}+Lv\leq0$ in $Q_{r}:=Q_{r}(Y),Y=(y,s)\in\mathbb{R}^{n+1},r>0$.
For arbitrary disks $D_{\rho}:=B_{\rho}(z)\times\{\tau\}$ and
$D^{0}:=B_{\frac{r}{2}}(y)\times\{\sigma\}$, such that $B_{\rho}(z)\subset B_{r}(y)$
and
\begin{equation}
s-r^{2}\leq\tau<\tau+h^{2}r^{2}\leq\sigma\leq s,\label{eq:23}
\end{equation}
where $h\in(0,1)$ is a constant. Then we conclude that
\begin{equation}
\inf_{D_{\rho}}v\leq\left(\frac{2r}{\rho}\right)^{\gamma}\inf_{D^{0}}v\label{eq:24}
\end{equation}
where $\gamma=\gamma(n,\nu,h,S)$.\end{lem}
\begin{proof}
Without loss of generality, we may assume $m:=\inf_{D_{\rho}}v>0$, $r=1$,
$z=0$, $\tau=0$, $\sigma=s=h^{2}$. So $D_{\rho}=B_{\rho}(0)\times\{0\}$.
We can apply an additional linear transformation along $t$-axis,
we can also reduce the proof to the case $h=1$. Now fix the integer
$k$ such that $2^{-k-1}<\rho\leq2^{-k}$, and for $j=0,1,\ldots$,
and we define $y^{j}:=y^{*}+2^{-j}(y-y^{*})$, $B^{j}:=B_{2^{-j}}(y^{j})$,
where $y^{*}:=\frac{\rho}{1-\rho}y$, $Y^{j}:=\left(y^{j},4^{-j}\right)$,
$Q^{j}:=Q_{2^{-j}}(Y^{j})$, $D^{j}:=B_{2^{-j-1}}(y^{j})\times\{4^{-j}\}$.
By construction, $0=y^{*}+\rho(y-y^{*})$, so that
\[
B_{\rho}(0),\, B^{j}\in\{B_{\theta}(y^{*}+\theta(y-y^{*});\,0\leq\theta\leq1\}.
\]
Then by the assumption, $B_{\rho}(0)\subset B_{1}(y)$ it follows
$\left|y\right|\leq1-\rho$, $\left|y-y^{*}\right|\leq1$, and
\[
B^{k+1}\subset B_{\rho}(0)\subset B^{k}\subset B^{k-1}\subset\ldots\subset B^{1}\subset B^{0}=B_{1}(y).
\]
Apply Theorem \ref{thm:(Second-Growth-Theorem).} to the function
$u=1-\frac{1}{m}v$ in $Q^{k}$ with
\[
r=2^{-k},\,\rho=2^{-k-1},\, Y=Y^{k},\, z=0,\,\tau=0.
\]
Then we conclude that
\[
\sup_{D_{k}}u\leq\sup_{Q_{2^{-k-1}(Y^{k})}}u\leq\beta_{3}\sup_{Q^{k}}u\leq\beta_{3}=\beta_{2}(n,\nu,S,\frac{1}{2})<1,
\]
which is equivalent to
\[
\inf_{D_{\rho}}v=m\leq(1-\beta_{3})^{-1}\inf_{D^{k}}v=2^{\gamma}\inf_{D^{k}}v,
\]
where $\gamma:=-\log_{2}(1-\beta_{3})>0$. Similarly, if $k\geq1$,
we also have
\[
\inf_{D^{j}}v\leq 2^{\gamma}\inf_{D^{j-1}}v,
\]
for $j=1,2,\ldots,k$. Finally we obtain
\[
\inf_{D_{\rho}}v\leq2^{\gamma}\inf_{D^{k}}v\leq2^{2\gamma}\inf_{D^{k-1}}v\leq\ldots\leq2^{(k+1)\gamma}\inf_{D^{0}}v\leq\left(\frac{2r}{\rho}\right)^{\gamma}\inf_{D^{0}}v.
\]

\end{proof}

\section{Third Growth Theorem}

The first growth theorem tells us if $\mu\rightarrow0^{+}$ then $\beta_{1}\rightarrow0^{+}$.
The third growth theorem tells us if we have a nice control of the
measure of the set $\left\{ u>0\right\} $ near the bottom, then we
can have a more precise estimate. In other words, if we have the similar measure
condition for
\[
 Q^{0}:=Q_{\frac{r}{2}}(Y^{0}), Y^{0}=\left(y,s-\frac{3}{4}r^{2}\right).
\]
Then if $\mu<1$, then $\beta_{1}<1$. In order to carry out the third
growth theorem, we need a covering lemma. We proceed as in
\cite{FS}.
\begin{lem}
\label{lem:covering}Let a constant $\mu_{0}\in(0,1)$ be a fixed
number. For an arbitrary measurable set $\Gamma\subset\mathbb{R}^{n+1}$
with finite Lebesgue measure $\left|\Gamma\right|$, we introduce
the family of cylinders
\begin{equation}
\mathcal{A}:=\left\{ Q=Q_{r}(Y):\,\left|Q\cap\Gamma\right|\geq(1-\mu_{0})\left|Q\right|\right\} .\label{eq:31}
\end{equation}
Then the open set $E:=\cup_{Q\in\mathcal{A}}Q$ satisfies,
\begin{equation}
\left|\Gamma\backslash E\right|=0,\label{eq:32}
\end{equation}
and
\begin{equation}
\left|E\right|\geq q_{0}\left|\Gamma\right|,\,\, q_{0}:=1+3^{-n-1}\mu_{0}>1.\label{eq:34}
\end{equation}
\end{lem}
\begin{proof}
From the fact that almost every point of $\Gamma$ is a point of density,
we have $\left|\Gamma\backslash E\right|=0$. More precisely, suppose
$\left|\Gamma\backslash E\right|>0$, then we can choose a cylinder
$Q^{*}:=B_{\frac{1}{m}}(y)\times(s-\frac{1}{m},s)$, with some $m\in\mathbb{N}$
such that
\begin{equation}
\left|Q^{*}\cap\left(\Gamma\backslash E\right)\right|\geq(1-\mu_{0})\left|Q^{*}\right|.\label{eq:35}
\end{equation}
Notice that $Q^{*}$ is a union of $m$ disjoint parabolic cylinders
\[
Q_{k}^{*}:=B_{\frac{1}{m}}(y)\times\left(s-\frac{k+1}{m^{2}},s-\frac{k}{m^{2}}\right),\,\, k=0,1,\ldots,m-1,
\]
therefore the above inequality \eqref{eq:35} must be true for some $Q_{k}^{*}$
instead of $Q^{*}$. Then $Q_{k}^{*}\in\mathcal{A}$, $Q_{k}^{*}\subset E$,
$Q_{k}^{*}\cap\left(\Gamma\backslash E\right)$ is empty and of course
the above inequality can not be true for $Q_{k}^{*}$. Therefore we
have $\left|\Gamma\backslash E\right|=0$.

Now, for each $Q=Q_{r}(Y)\in\mathcal{A}$ with $\left|Q\cap\Gamma\right|\geq(1-\mu_{0})\left|Q\right|$,
we continuously increase $r$ such that we achieve the exact equality,
i.e., $\left|Q\cap\Gamma\right|=(1-\mu_{0})\left|Q\right|$. Therefore,
we can write
\[
E:=\cup_{Q\in A_{0}}Q,\,\,\,\mathcal{A}_{0}:=\left\{ Q=Q_{r}(Y):\,\left|Q\cap\Gamma\right|=(1-\mu_{0})\left|Q\right|\right\} .
\]
Next, we follow the well-known argument in the classical Vitali covering
lemma with parabolic cylinders instead of balls or cubes. We construct
an at most countable sequence of cylinders $Q^{k}$, $k=1,2,\ldots$
as follows: we denote
\[
R_{1}:=\sup\left\{ r:\,\, Q_{r}(Y)\in\mathcal{A}_{0}\right\} .
\]
By an easy compactness argument gives us that this supremum is obtained
for some cylinder $Q_{R_{1}}(Y_{1})\in\mathcal{A}_{0}$. Define $Q^{1}:=Q_{R_{1}}(Y_{1})$.
Now assume that $Q^{i}:=Q_{R_{i}}(Y_{i})$ for $i=1,2,\ldots,k$ have
been selected for some $k\geq1$, we set
\[
\mathcal{A}_{k+1}:=\left\{ Q=Q_{r}(Y)\in\mathcal{A}_{0}:\, Q\cap Q^{i}=\emptyset,\, i=1,2,\ldots,k\right\} .
\]
If $\mathcal{A}_{k+1}$ is nonempty, then we denote
\[
R_{k+1}:=\sup\left\{ r:\,\, Q_{r}(Y)\in\mathcal{A}_{k+1}\right\}
\]
this supremum is obtained for $Q^{k+1}:=Q_{R_{k+1}}(Y_{k+1})\in\mathcal{A}_{k+1}$.

In the case when all of the sets $\mathcal{A}_{0}\supset\mathcal{A}_{1}\supset\mathcal{A}_{2}\supset\ldots$
are nonempty, we get a countable sequence of cylinder $Q^{i}=Q_{R_{i}}(Y_{i})$,
$i=1,2,\ldots.$ If however $\mathcal{A}_{i}\neq\emptyset$ for $i=1,2,\ldots,k$
and $\mathcal{A}_{k+1}=\emptyset$, then we have a finite sequence
of cylinders $Q^{i},$$i=1,2,\ldots,k$. In the latter case we set,
by definition, $R_{k+1}=R_{k+1}=\cdots=0$. Clearly, by our construction,
the cylinders $Q^{i}$ are pairwise disjoint, $R_{1}\geq R_{2}\geq\cdots$,
and $R_{i}\rightarrow0$ as $i\rightarrow\infty$.

Take an arbitrary cylinder $Q_{r}(Y)\in\mathcal{A}_{0}$. We have
$R_{1}\geq R_{2}\geq\cdots\geq R_{k}\geq r>R_{k+1}$ for some integer
$k\geq1$. Since $r>R^{k+1}$, the cylinder $Q_{r}(Y)$ does not belong
to $\mathcal{A}_{k+1}$ and therefore $Q_{r}(Y)\cap Q^{i}\neq\emptyset$
for some $i\leq k$. Since $r<R_{i},$ this implies $Q_{r}(Y)\subset\tilde{Q^{i}}$,
where
\[
\tilde{Q}:=B_{3\rho}\times(s-2\rho^{2},s+\rho^{2})
\]
for
\[
Q=Q_{\rho}(z,s)=B_{\rho}\times(s-\rho^{2},s).
\]
Therefore for arbitrary $Q_{r}(Y)\in\mathcal{A}_{0}$ is a subset
of $\tilde{Q^{i}}$ for some $i\geq1$. Then we have
\[
E\subset\cup_{i}\tilde{Q^{i}},\,\,\left|E\right|\leq\sum_{i}\left|\tilde{Q^{i}}\right|=3^{n+1}\sum_{i}\left|Q^{i}\right|.
\]
On the other hand since $Q^{i}\in\mathcal{A}_{0}$ are pairwise disjoint,
\[
\left|E\backslash\Gamma\right|\geq\sum_{i}\left|\left(E\backslash\Gamma\right)\cap Q^{i}\right|=\mu_{0}\sum_{i}\left|Q^{i}\right|.
\]
Then with the above two relations, we obtain
\[
\frac{\left|E\right|}{\left|\Gamma\right|}=1+\frac{\left|E\backslash\Gamma\right|}{\left|\Gamma\right|}\geq1+\frac{\left|E\backslash\Gamma\right|}{\left|E\right|}\geq1+\frac{\mu_{0}}{3^{n+1}}=q_{0}>1.
\]
Therefore
\[
\left|E\right|\geq q_{0}\left|\Gamma\right|.
\]

\end{proof}
Now we consider the change of measures after we shift cylinders with
respect to certain rule, which will be helpful when we prove the third
growth theorem. The lemma below seems trivial but we need to be cautious.
The way a cylinder shifted depends on the size and the position
of the cylinder,so some originally overlapped cylinders might be
disjoint after we shift them, or vice versa.
\begin{lem}
\label{lem:shift}For a fixed constant $K_{1}>1$ and any standard
cylinder $Q=Q_{r}(Y)=B_{r}(y)\times(s-r^{2},s)$, we denote $\hat{Q}:=B_{r}(y)\times(s+r^{2},s+K_{1}r^{2})$.
Then for an arbitrary family of standard cylinders, the $n+1$ dimensional
Lebesgue measures of the sets
\begin{equation}
E:=\cup_{Q\in\mathcal{A}}Q\label{eq:36}
\end{equation}
 and
\begin{equation}
\hat{E}:=\cup_{Q\in\mathcal{A}}\hat{Q}\label{eq:37}
\end{equation}
have the following quantitative relation:
\begin{equation}
\left|\hat{E}\right|\geq q_{1}\left|E\right|,\,\, q_{1}:=\frac{K_{1}-1}{K_{1}+1}.\label{eq:38}
\end{equation}
\end{lem}
\begin{proof}
By Fubini's theorem, we obtain
\[
\left|E\right|=\int\left|E_{x}\right|dx,\,\,\,\left|\hat{E}\right|=\int\left|\hat{E}_{x}\right|dx
\]
where we use the standard notation for $x\in\mathbb{R}^{n},$
\[
E_{x}:=\left\{ t\in\mathbb{R}:\,(x,t)\in E\right\} ,\,\,\,\hat{E}_{x}:=\left\{ t\in\mathbb{R}:\,(x,t)\in\hat{E}\right\} .
\]
Now we see it suffices to show $\left|\hat{E}_{x}\right|\geq q_{1}\left|E_{x}\right|$, $\forall x\in\mathbb{R}^{n}$.
Then everything is reduced to one-dimensional topology.

Now fix an $x$ such that $E_{x}$ is not empty. Then for this fixed
$x$ the open set $\hat{E}_{x}$ is a union of disjoint open intervals
$\hat{I}_{k}$, $k=1,2,\ldots,$. Here we use the basic fact from
1-D topology about the structure of open sets. If $t\in E_{x}$, then
$(x,t)\in Q_{r}=B_{r}(y)\times(s-r^{2},s)$ for some cylinder $Q_{r}\in\mathcal{A}$,
and $(s+r^{2},s+K_{1}r^{2})\subset\hat{I}_{k}$ for some $k$. Then
clearly, $r^{2}\leq r_{k}^{2}:=(K_{1}-1)^{-1}\left|\hat{I}_{k}\right|$.
We can also choose $s_{k}$ such that $\hat{I}_{k}=(s_{k}+r_{k}^{2},s_{k}+K_{1}r_{k}^{2})$
by our construction. Then, we observe that
\[
s_{k}+r_{k}^{2}\leq s+r^{2},\,\, s_{k}-r_{k}^{2}\leq s-r^{2}.
\]
The first inequality is trivial, and the second one follows from the
fact $r_{k}\ge r$ by the construction. Therefore, 
\[
t\in(s-r^{2},s)\subset J_{k}:=(s_{k}-r_{k}^{2},s_{k}+K_{1}r_{k}^{2}).
\]
Hence we get
\[
E_{x}\subset\cup J_{k},
\]
which implies
\[
q_{1}\left|E_{x}\right|\leq q_{1}\left|\cup J_{k}\right|\leq q_{1}\sum_{k}\left|J_{k}\right|=\sum_{k}\left|\hat{I}_{k}\right|=\left|\hat{E}_{x}\right|.
\]
So we conclude that
\[
\left|\hat{E}\right|\geq q_{1}\left|E\right|.
\]
\end{proof}
\begin{thm}[Third Growth Theorem]
\label{thm:(Third-Growth-Theorem).}Let a
function $u\in C^{2,1}\left(\overline{Q_{r}}\right)$, where $Q_{r}=Q_{r}(Y)$,
$Y=(y,s)\in\mathbb{R}^{n+1}$, $r>0$, and let $-u_{t}+Lu\geq0$ in
$Q_{r}$. In addition, we assume
\begin{equation}
\left|\left\{ u>0\right\} \cap Q^{0}\right|\leq\mu\left|Q^{0}\right|,\label{eq:39}
\end{equation}
where
\begin{equation}
Q^{0}:=Q_{\frac{r}{2}}(Y^{0}),\,\, Y^{0}=\left(y,s-\frac{3}{4}r^{2}\right)\label{eq:310}
\end{equation}
and $\mu<1$ is a constant. Then we have
\begin{equation}
\mathcal{M}_{\frac{r}{2}}(Y)\leq\beta\mathcal{M}_{r}(Y)\label{eq:311}
\end{equation}
with a constant
\[
\beta:=\beta(n,\nu,S,\mu)<1.
\]
\end{thm}
\begin{proof}
Without loss of generality, we may rescale and translate our problem so that  $r=2$, $Y=(y,s)=(0,0)$. Now under this
setting, $Q^{0}=B_{1}(0)\times(-4,-3)$ and $\left|Q^{0}\right|=\left|B_{1}(0)\right|$, 
which only depends on $n$.

Now consider $\Gamma:=\left\{ u\leq0\right\} \cap Q^{0}$. Then from
the above measure condition, we have
\begin{equation}
\left|\Gamma\right|\geq(1-\mu)\left|Q^{0}\right|=(1-\mu)\left|B_{1}(0)\right|=:c_{0}=c_{0}(\mu,n)>0.\label{eq:312}
\end{equation}
From the first growth theorem, we know the constant $\beta_{1}=\beta_{1}(n,\nu,S,\mu)\rightarrow0^{+}$
as $\mu\rightarrow0^{+}$. So we can find a constant $\mu_{0}=\mu_{0}(n,\nu,S)\in(0,1)$
such that the first growth theorem holds for a constant $\beta_{1}\leq\frac{1}{2}$.
With this constant $\mu_{0}$ and $\Gamma$, we perform the covering
lemma, Lemma \ref{lem:covering}, to obtain a family of cylinders $\mathcal{A}$ defined
by the formula in the above lemma. Then by the results from the above
lemmas, if we denote $E:=\cup_{Q\in\mathcal{A}}Q$, then we obtain
\[
\left|\Gamma\backslash E\right|=0,
\]
and
\begin{equation}
\left|E\right|\geq q_{0}\left|\Gamma\right|,\,\, q_{0}:=1+3^{-n-1}\mu_{0}>1.\label{eq:313}
\end{equation}
Denote $\epsilon_{0}:=3^{-n-2}\mu_{0}$, i.e., $q_{0}=1+3\epsilon_{0}$.
Now choose constant $K_{1}>0$ such that
\begin{equation}
q_{1}:=\frac{K_{1}-1}{K_{1}+1}=\frac{1+2\epsilon_{0}}{1+3\epsilon_{0}},\label{eq:314}
\end{equation}
i.e.,
\[
K_{1}=K_{1}(n,\nu,S)=5+2\epsilon_{0}^{-1}.
\]
By the above two lemmas, we conclude that
\begin{equation}
\left|\hat{E}\right|\geq q_{1}\left|E\right|\geq q_{0}q_{1}\left|\Gamma\right|=\left(1+2\epsilon_{0}\right)\left|\Gamma\right|.\label{eq:315}
\end{equation}
In order to have some estimate of the size of the cylinders in the
family $\mathcal{A}$, we introduce another cylinder $Q^{1}$ such
that
\[
Q^{0}\subset Q^{1},\,\,\left|Q^{1}\backslash Q^{0}\right|\leq\epsilon_{0}c_{0},\,\, dist\left(Q^{0},\partial Q^{1}\right)\geq c_{1}=c_{1}(n,\nu,\mu).
\]
Then there are two possibilities about $\hat{E}$ and $Q^{1}$. (a)
$\hat{E}\backslash Q^{1}\neq\emptyset$ and (b) $\hat{E}\backslash Q^{1}=\emptyset$.

(a) If $\hat{E}\backslash Q^{1}\neq\emptyset$, this can only be true
if there are some cylinders $Q\in\mathcal{A}$ which are large enough,
i.e., there exists $Q=Q_{r}(Y)\in\mathcal{A}$ with $r\geq r_{0}=r_{0}(n,\nu,\mu)>0$.
Notice that
\begin{equation}
\left|\left\{ u>0\right\} \cap Q\right|\leq\left|Q\backslash\Gamma\right|\leq\mu_{0}\left|Q\right|,\label{eq:316}
\end{equation}
 for $Q\in\mathcal{A}$. By the first growth theorem and the choice
of $\mu_{0}$, we have
\begin{equation}
\sup_{Q_{\frac{r}{2}}(Y)}u\leq\frac{1}{2}\sup_{Q_{r}(Y)}u_{+}\leq\frac{1}{2}M,\label{eq:317}
\end{equation}
for $Q=Q_{r}(Y)\in A$ where
\[
M:=\sup_{Q_{r}(0)}u_{+}.
\]
Therefore,
\[
u-\frac{1}{2}M\leq0
\]
on $D:=B_{\frac{r}{2}}(Y)\times\{s\}$ for $Q_{r}(Y)\in\mathcal{A}$,
$Y=(y,s)$. Now we apply the second growth theorem to $u-\frac{1}{2}M$
with $\rho=\frac{1}{2}r_{0}$. By the theorem, we obtain $u-\frac{1}{2}M\in\mathcal{M}(\beta_{2},0,2)$
with $\beta_{2}=\beta_{2}(n,\nu,S,\mu)<1$. Then
\begin{eqnarray*}
\sup_{Q_{1}}u & = & \frac{1}{2}M+\sup_{Q_{1}}\left(u-\frac{1}{2}M\right)\\
 & \leq & \frac{1}{2}M+\beta_{2}\sup_{Q_{2}}\left(u-\frac{1}{2}M\right)_{+}\\
 & = & \frac{1}{2}\left(1+\beta_{2}\right)M,
\end{eqnarray*}
and $u\in\mathcal{M}(\beta_{0},0,2)$ with $\beta_{0}=\frac{1}{2}\left(1+\beta_{2}\right)<1.$

(b) If $\hat{E}\backslash Q^{1}=\emptyset$, then $\hat{E}\subset Q^{1}$,
and by the measure relations, if we set $\Gamma_{1}:=\hat{E}\cap Q^{0}$, then $\Gamma_{1}$
satisfies
\begin{equation}
\left|\Gamma_{1}\right|=\left|\hat{E}\cap Q^{0}\right|\geq\left|\hat{E}\right|-\left|Q^{1}\backslash Q^{0}\right|\geq\left(1+2\epsilon_{0}\right)\left|\Gamma\right|-\epsilon_{0}c_{0}\geq\left(1+\epsilon_{0}\right)\left|\Gamma\right|.\label{eq:318}
\end{equation}
From the above argument in case (a), we know that $\forall Q\in\mathcal{A}$,
we have $u\leq\beta_{0}M$ on $\hat{Q}$ with $\beta_{0}=\beta_{0}(n,\nu,S,\mu)<1$.
Since the sets $\hat{Q}$ cover $\Gamma_{1}$, we know $u\leq\beta_{0}M$
on $\Gamma_{1}$. Therefore,
\begin{equation}
\left|\left\{ u\leq\beta_{0}M\right\} \cap Q^{0}\right|\geq\left(1+\epsilon_{0}\right)\left|\left\{ u\leq0\right\} \cap Q^{0}\right|.\label{eq:319}
\end{equation}
Now we have proved that either (a) $u\leq\beta_{0}M$ on $Q_{1}$,
or in case (b) $u$ satisfies the measure relation $\left|\left\{ u\leq\beta_{0}M\right\} \cap Q^{0}\right|\geq\left(1+\epsilon_{0}\right)\left|\left\{ u\leq0\right\} \cap Q^{0}\right|$.
Now we proceed recursively starting from
\begin{equation}
u_{0}:=u,\,\, M_{0}:=M,\label{eq:320}
\end{equation}
and define
\begin{equation}
u_{k+1}:=u_{k}-\beta_{0}M_{k},\,\, M_{k+1}:=\sup_{Q_{r}}u_{k+1}\label{eq:321}
\end{equation}
for $k=0,1,2,\ldots.$ Then we can easily derive that $\forall k$
\begin{equation}
u_{k}=u-\left[1-\left(1-\beta_{0}\right)^{k}\right]M,\,\, M_{k}=\left(1-\beta_{0}\right)^{k}M.\label{eq:323}
\end{equation}
If case (b) holds for all $u_{k}$ with $k=0,1,2\ldots,m-1$, then
\begin{eqnarray*}
\left|Q^{0}\right| & \geq & \left|\left\{ u_{m}\leq0\right\} \cap Q^{0}\right|\\
 & = & \left|\left\{ u_{m-1}\leq\beta_{0}M_{m-1}\right\} \cap Q^{0}\right|\\
 & \geq & \left(1+\epsilon_{0}\right)\left|\left\{ u_{m-1}\leq0\right\} \cap Q^{0}\right|\\
 & \geq\\
 & \vdots\\
 & \geq & \left(1+\epsilon_{0}\right)^{m}\left|\left\{ u_{0}\leq0\right\} \cap Q^{0}\right|\\
 & \geq & \left(1+\epsilon_{0}\right)^{m}c_{0}.
\end{eqnarray*}
If now, we take $m\in\mathbb{N}$ such that $\left(1+\epsilon_{0}\right)^{m}c_{0}>\left|Q^{0}\right|$,
then (b) fails for some $u_{k}$ with $k\leq m-1$. Therefore, we
have
\[
u_{m}\leq u_{k+1}\leq0,
\]
and
\begin{equation}
u\leq\left[1-\left(1-\beta_{0}\right)^{m}\right]M\label{eq:325}
\end{equation}
on $Q_{1}$. We can conclude that 
\[
u\in\mathcal{M}(\beta,0,2)
\]
with
\begin{equation}
\beta:=1-\left(1-\beta_{0}\right)^{m}<1.\label{eq:326}
\end{equation}
\end{proof}
\begin{cor}
\label{lower}Let a function $v\in C^{2,1}\left(\overline{Q_{r}}\right)$
be such that $v\geq0$, and $-v_{t}+Lv\leq0$ in $Q_{r}$, and
\begin{equation}
\left|\left\{ v\geq1\right\} \cap Q^{0}\right|>(1-\mu)\left|Q^{0}\right|.\label{eq:327}
\end{equation}
Then
\begin{equation}
v\geq1-\beta>0\label{eq:328}
\end{equation}
on $Q_{\frac{r}{2}}$, where $\beta=\beta(n,\nu,\mu,S)<1$ for $\mu<1$.
\end{cor}
The corollary basically tells us quantitatively that if $v$ is large in
a large region of a cylinder, then $v$ is large in half of the cylinder.
We can notice the above corollary is much more precise
than the intermediate comparison results in section for the first
growth thereom.
\begin{proof}
We can rescale our problem again, and assume $r=2$. Then the function
$u=1-v$ satisfies $-u_{t}+Lu\geq0$ in $Q_{2}$, and
\begin{eqnarray*}
\left|\left\{ u>0\right\} \cap Q^{0}\right| & = & \left|\left\{ v<1\right\} \cap Q^{0}\right|=\left|Q^{0}\right|-\left|\left\{ v\geq1\right\} \cap Q^{0}\right|\\
 &  & \leq\left|Q^{0}\right|-(1-\mu)\left|Q^{0}\right|\leq\mu\left|Q^{0}\right|.
\end{eqnarray*}
Then we can apply the third growth theorem to $u$, we obtain
\[
\sup_{Q_{1}}u\leq\beta\sup_{Q_{2}(X_{0})}u_{+}\leq\beta,
\]
and
\[
\inf_{Q_{1}}v=1-\sup_{Q_{1}}u\geq1-\beta.
\]

\end{proof}

\section{Interior Harnack Inequality}

With the first growth theorem, we can do the following useful argument,
which is helpful for us to find a non-degenerate point to build a
bridge between two regions we are interested in. Without loss of generality,
we still assume $r=1$, for $X\in Q_{1}(Y)$, we define
\begin{equation}
d(X):=\sup\left\{ \rho>0:\, Q_{\rho}(X)\subset Q_{1}(Y)\right\} .\label{eq:41}
\end{equation}
Roughly here $d$ plays roles of weights with which we can make
sure the point we are interested in is not degenerate, i.e., it is
in the interior of the cylinder. For $\gamma>0$,
we consider $d^{\gamma}u(x)$ instead of $u(x)$. $d^{\gamma}u(x)$
is a continuous function in $\overline{Q_{1}(Y)}$. Clearly, $d(Y)=1$,
we have
\begin{equation}
u(Y)=d^{\gamma}u(Y)\leq M:=\sup_{Q_{1}(Y)}d^{\gamma}u.\label{eq:42}
\end{equation}
By our construction, $d^{\gamma}u$ vanishes on $\partial_{p}Q_{1}$,
so $\exists X_{0}\in\overline{Q_{1}(Y)}\backslash\partial_{p}Q_{1}$
such that
\begin{equation}
M=d^{\gamma}u(X_{0}).\label{eq:45}
\end{equation}
Let $r_{0}:=\frac{1}{2}d(X_{0})$, we consider the intermediate region
$Q_{r_{0}}(X_{0})$. In this region, we have
\[
\forall X\in Q_{r_{0}}(X_{0}),\,\, d(X)\geq r_{0}.
\]
Therefore, we can conclude that
\begin{equation}
\sup_{Q_{r_{0}}(X_{0})}u\leq r_{0}^{-\gamma}\sup_{Q_{r_{0}}(X_{0})}d^{\gamma}u\leq r_{0}^{-\gamma}M\leq2^{\gamma}u(X_{0}).\label{eq:46}
\end{equation}
Now, we define $v=u-\frac{1}{2}u(X_{0})$, then
\[
v(X_{0})=\frac{1}{2}u(X_{0})\geq2^{-1-\gamma}\sup_{Q_{r_{0}}(X_{0})}u>2^{-1-\gamma}\sup_{Q_{r_{0}}(X_{0})}v.
\]
From Theorem \ref{thm:(First-Growth-Theorem)},
we know $\exists\mu(n,\nu,\gamma,S)\in(0,1]$ such that the first
growth theorem holds with $\beta_{1}=2^{-1-\gamma}$. Now the above
inequality tells us that $v$ does not satisfy the measure condition
in the first growth theorem. So
\begin{equation}
\left|\left\{ v>0\right\} \cap Q_{r_{0}}(X_{0})\right|=\left|\left\{ u>\frac{1}{2}u(X_{0})\right\} \cap Q_{r_{0}}(X_{0})\right|>\mu\left|Q_{r_{0}}(X_{0})\right|.\label{eq:47}
\end{equation}
With the above preparation, we  are ready to prove the interior
Harnack inequality.
\begin{thm*}[Interior Harnack Inequality]
 Suppose $u\in C^{2,1}(Q_{2r}(Y)) \cap C(\overline{Q_{2r}(Y)})$ and $-u_{t}+Lu=0$ in $Q_{2r}(Y)$,
$Y=(y,s)\in\mathbb{R}^{n+1}$ and $r>0$. If $u\geq0$, then
\begin{equation}
\sup_{Q^{0}}u\leq N\inf_{Q_{r}}u,\label{eq:48}
\end{equation}
where $N=N(n,\nu,S)$ and $Q^{0}=B_{r}(y)\times(s-3r^{2},s-2r^{2})$.
\end{thm*}
We will build a non-degenerate intermediate region to get a quantitative
relation between two regions we are interested in with the help of three growth theorems.
\begin{proof}
After rescaling and translating as necessary,
we can assume $Y=0$ and $r=1$. Now $Q_{1}=B_{1}(0)\times(-1,0)$,
$Q^{0}=B_{1}(0)\times(-3,-2)$. It is easy to see that if we define
$d(X):=\sup\left\{ \rho>0:\, Q_{\rho}(X)\subset Q_{2}(0)\right\} $,
then $d(X)\geq1$ in $Q^{0}$. Hence, if we consider $Q^{1}:=B_{2}(0)\times(-3,-2)$
we have
\begin{equation}
\sup_{Q^{0}}u\leq M:=\sup_{Q^{1}}d^{\gamma}u,\label{eq:49}
\end{equation}
where $\gamma$ is chosen at the same as the $\gamma$ in Lemma \ref{lem:quotient}
with $h=\frac{1}{2}$. From the above discussion, we can find $\exists X_{0}\in\overline{Q^{1}}\backslash\left[\partial_{p}Q^{1}\cap\partial_{p}Q_{2}\right]$
such that
\begin{equation}
d^{\gamma}u(X_{0})=M.\label{eq:410}
\end{equation}
Similarly as above, we define
\begin{equation}
\rho=\frac{1}{4}d(X_{0})\in(0,\frac{1}{2}],\label{eq:411}
\end{equation}
and
\begin{equation}
Q_{0}=Q_{\rho}(X_{0})\cap\left\{ u>\frac{1}{2}u(X_{0})\right\} .\label{eq:412}
\end{equation}
By the above discussion, we conclude that
\[
\left|Q_{0}\right|>\mu_{1}\left|Q_{\rho}(X_{0})\right|
\]
for some constant $\mu_{1}=\mu_{1}(n,\nu,S,\gamma)>0$. Now we apply
Corollary \ref{lower} with
\[
v=\frac{2}{u(X_{0})}u,\,\, Q_{r}=Q_{2\rho}(Y_{0}),\,\, Y_{0}=(x_{0},t_{0}+3\rho^{2}),\,\, Q^{0}=Q_{\rho}(X_{0}),\,\,1-\mu=\mu_{1}.
\]
Then we have
\begin{equation}
u\geq\beta u(X_{0})\label{eq:415}
\end{equation}
on $Q_{\rho}(Y_{0})$ with $\beta=\beta(n,\nu,S)>0$. Next we apply
Lemma \ref{lem:quotient} with
\[
v=u,\,\, r=2,\,\, D_{\rho}=B_{\rho}(x_{0})\times\{t_{0}+2\rho^{2}\}\subset\overline{Q_{\rho}(Y_{0})},
\]
and
\[
D^{0}=B_{1}(0)\times\{\tau\},\,\,\forall\tau\in(-1,0).
\]
So we obtain
\begin{equation}
\beta u(X_{0})\leq\inf_{D_{\rho}}u\leq\left(\frac{4}{\rho}\right)^{\gamma}\inf_{Q_{1}(0)}u.\label{eq:414}
\end{equation}
Finally, with the help of the intermediate region, we conclude that
\begin{equation}
\sup_{Q^{0}}u\leq M=d^{\gamma}u(X_{0})=\left(4\rho\right)^{\gamma}u(X_{0})\leq\beta^{-1}4^{2\gamma}\inf_{Q_{1}(0)}u.\label{eq:416}
\end{equation}
Taking $N=N(n,\nu,S)=\beta^{-1}4^{2\gamma}$ gives the desired result.
\end{proof}
It is well-known that it is easy to derive the H\"older continuity
of solutions from the Harnack inequality by standard oscillation and
iteration arguments.
\begin{thm}\label{hoelder}
If $u\in W(Q_{r})$ and is a solution of $-u_{t}+Lu=0$ in $Q_{r}$, then u is H\"older
continuous in $Q_{\frac{r}{2}}$.
\end{thm}

\section{Approximation}

In all the proofs from above sections, we always assume $u$ is $C^{2,1}$.
In this section, we briefly show we can use an approximation argument
to show that all results hold for $u\in W(Q_{2r})=W_{n,\infty}^{2,1}(Q_{2r})\cap C(\overline{Q_{2r}})$,
where $u\in W_{n,\infty}^{2,1}(Q_{2r})$ means $u_{t},\, D_{i}u,\, D_{ij}u\in\left(L_{x}^{n}L_{t}^{\infty}\right)_{loc}$.
Throughout, we assume
\begin{equation}
u\geq0,\,\,-u_{t}+Lu=-u_{t}+\sum_{ij}a_{ij}D_{ij}u+\sum_{i}b_{i}D_{i}u=0\label{eq:51}
\end{equation}
in $Q_{2r}$. We can approximate $a_{ij}$, $b_{i}$ and $u$ by smooth
functions $a_{ij}^{\epsilon}\rightarrow a_{ij}$, $b_{i}^{\epsilon}\rightarrow b_{i}$
a.e. as $\epsilon\rightarrow0^{+}$. And $u^{\epsilon}\rightarrow u$
in $W_{n,\infty}^{2,1}$ as $\epsilon\rightarrow0^{+}$. Then
\begin{equation}
f^{\epsilon}=-u_{t}^{\epsilon}+L^{\epsilon}u^{\epsilon}=-u_{t}^{\epsilon}+\sum_{ij}a_{ij}^{\epsilon}D_{ij}u^{\epsilon}+\sum_{i}b_{i}^{\epsilon}D_{i}u^{\epsilon}\rightarrow0\label{eq:52}
\end{equation}
in $\left(L_{x}^{n}L_{t}^{\infty}\right)_{loc}(Q_{2r})$ as $\epsilon\rightarrow0^{+}$.
With the existence of solution to for equations with smooth coefficients, therefore
we can write
\[
u^{\epsilon}=v^{\epsilon}+w^{\epsilon},
\]
where
\[
-v_{t}^{\epsilon}+L^{\epsilon}v^{\epsilon}=0
\]
 in $Q_{2r}$ and
\[
v^{\epsilon}=u^{\epsilon}
\]
 on $\partial_{p}Q_{2r}$:
\[
-w_{t}^{\epsilon}+L^{\epsilon}w^{\epsilon}=f^{\epsilon},
\]
\[
w^{\epsilon}=0
\]
 on $\partial_{p}Q_{2r}$. By the Alexandrov-Bakelman-Pucci-Krylov-Tso
estimate, we know $w^{\epsilon}\rightarrow0$ in $L^{\infty}$ and
$v^{\epsilon}$ satisfies the Harnack inequality. Finally, by
an easy limiting argument, $u$ also satisfies the Harnack inequality.
\begin{rem}
For other values for $p,\, q$ for $q<\infty$, the approximation
argument for $u\in W_{p,q}^{2,1}(Q_{2r})\cap C(\overline{Q_{2r}})$
is similar but actually easier from standard results about $L^{p}$ spaces.
\end{rem}

\section{Applications}

This section, we show some applications of the interior Harnack inequality.
We will just formulate some results based on the interior Harnack inequality. In particular the
boundary Harnack inequality, the boundary backward Harnack inequality, and the H\"older continuity of quotients. The detailed proofs are
provided \cite{FSY}. And one can find more details on applications
of the interior Harnack inequality in \cite{FSY,S98}. We start with
some additional basic notations in order to formulate our results.

For $X=(x,t)\in\mathbb{R}^{n+1}$ and $r>0$, a standard cylinder
is defined as
\[
Q_{r}(X)=Q_{r}(x,t)=B_{r}(x)\times(t-r^{2},t),
\]
where $B_{r}(x)=\left\{ y\in\mathbb{R}^{n},\,|y-x|<r\right\} $. For
a constant $\delta>0$,$\Omega\in\mathbb{R}^{n}$, $Q_{\Omega}:=\Omega\times(0,T)$,
we define
\begin{equation}
\Omega^{\delta}=\left\{ x\in\Omega:\, dist(x,\partial\Omega)>\delta\right\} =\left\{ x\in\Omega:\,\overline{B_{\delta}}(x)\subset\Omega\right\} ,\label{eq:61}
\end{equation}
\begin{equation}
Q_{\Omega}^{\delta}=\Omega^{\delta}\times(\delta^{2},T)=\left\{ X\in Q_{\Omega}:\,\overline{Q_{\delta}}(X)\subset Q_{\Omega}\right\} .\label{eq:62}
\end{equation}

We assume $\Omega$ to be a bounded Lipschitz domain in $\mathbb{R}^{n}$.
By a Lipschitz domain, we mean there are positive constants $r_{\Omega}$
and $m_{\Omega}$ such that $\forall y\in\partial\Omega$, we can
find an orthonormal frame centered at $y$, in which we have
\begin{equation}
\Omega\cap B_{r_{\Omega}}(y)=\left\{ x=\left(x',x_{n}\right):\, x'\in\mathbb{R}^{n-1},\, x_{n}>\phi(x'),\,\left|x\right|<r_{\Omega}\right\} ,\label{eq:63}
\end{equation}
and
\begin{equation}
\left\Vert \nabla\phi\right\Vert _{L^{\infty}}\leq m_{\Omega}.\label{eq:64}
\end{equation}
Also in such coordinates, $y\in\partial\Omega$ is represented as
$(0,0)$ and $(0,r)\in\Omega$ for all $r\in(0,r_{\Omega}]$. For
$Q_{\Omega}=\Omega\times(0,T)$, $Y=(y,s)=(0,0,s)\in\partial_{x}Q_{\Omega}=\partial_{x}\Omega\times(0,T)$,
and for $r>0$, we denote
\begin{equation}
\overline{Y_{r}}=(0,r,s+r^{2}),\,\,\,\underline{Y_{r}}=(0,r,s-2r^{2}),\label{eq:65}
\end{equation}
For fixed positive constants $r_{0}$, $R_{0}$ and $M_{0}$, we assume
our domain $\Omega$ satisfies
\begin{equation}
r_{\Omega}\geq r_{0},\,\, m_{\Omega}\leq M_{0},\,\, diam(\Omega)\leq R_{0}.\label{eq:66}
\end{equation}
In the most general form, we can state the interior Haranck principle
as the following:
\begin{thm}
(Harnack Principle). Suppose $u\in W$ in the sense of Definition \ref{space}\eqref and $u\geq0$ satisfy $-u_{t}+Lu=0$ in a bounded
domain $Q_{\Omega}=\Omega\times(0,T)$, constant $\delta>0$ such
that $\Omega^{\delta}$ is connected set, $X=(x,t)$, $Y=(y,s)\in Q_{\Omega}^{\delta}$,
and $s-t\geq\delta^{2}$. Then
\begin{equation}
u(X)\leq Nu(Y),\label{eq:67}
\end{equation}
where the constant $N=N(n,\nu,S,R_{0},T,\delta)$.
\end{thm}
Now we can state the theorem on a boundary Harnack inequality. This
kind of inequality is also called Carleson type inequality.
\begin{thm}
\label{thm:bdy harnack}Let $Y=(y,s)\in\partial_{x}Q_{\Omega}$ and
$0<r<\frac{1}{2}\min\left(r_{0},\sqrt{T-s},\sqrt{s}\right)$ be fixed.
Then for any non-negative solution $-u_{t}+Lu=0$ in $Q_{\Omega}=\Omega\times(0,T)$,
which continuously vanishes on $\Gamma=\partial_{x}Q_{\Omega}\cap Q_{2r}(Y)$,
we have
\begin{equation}
M_{0}=\sup_{Q_{\Omega}\cap Q_{2r}(Y)}d_{0}^{\gamma}u\leq Nr^{\gamma}u(\overline{Y_{r}}),\label{eq:68}
\end{equation}
where
\[
d_{0}=d_{0}(X):=\sup\left\{ \rho>0:\, Q_{\rho}(X)\subset Q_{2\rho}(Y)\right\} ,
\]
and $\gamma$ and $N$ are positive constants depending only on $n$,
$\nu,$ $S$ and $m_{\Omega}$. In particular
\begin{equation}
\sup_{Q_{\Omega}\cap Q_{r}(Y)}u\leq Nu(\overline{Y_{r}}).\label{eq:69}
\end{equation}

\end{thm}
Again, one can find the detailed proof in \cite{FSY}. With our growth
theorems and interior Haranck inequality, the remaining steps in the
proof are more or less independent of the specific structure of the
equations.

We also state a elliptic-type Harnack inequality.
\begin{thm}
\label{thm:elliptic type}Let $u$ be a non-negative solution $-u_{t}+Lu=0$
in $Q_{\Omega}=\Omega\times(0,T)$ which continuously vanishes on
$\partial_{x}Q_{\Omega}$, and let $0<\delta\leq\frac{1}{2}\min\left(r_{0},\sqrt{T}\right)$.
Then there exists a positive constant $N=N(n,\nu,S,m_{\Omega},r_{0},R_{0},T,\delta)$,
such that
\begin{equation}
\sup_{Q_{\Omega}^{\delta}}\leq N\inf_{Q_{\Omega}^{\delta}}u.\label{eq:610}
\end{equation}

\end{thm}
Next, the boundary backward Haranck inequality is formulated as follows:
\begin{thm}
\label{thm:bdybh}Let $u$ be a non-negative solution $-u_{t}+Lu=0$
in $Q_{\Omega}=\Omega\times(0,T)$ which continuously vanishes on
$\partial_{x}Q_{\Omega}$, and let $\delta>0$ be a constant. Then
there exists a positive constant $N=N(n,\nu,S,M_{0},r_{0},R_{0},T,\delta)$,
such that
\begin{equation}
u(x,s)\leq Nu(x,t)\label{eq:611}
\end{equation}
 where $T>s\geq t\geq s-d^{2}\geq\delta^{2}>0$, $d=dist(x,\partial\Omega)$.
\end{thm}
Again interested readers can find details in \cite{FSY}.

Finally, we state a result related to the H\"older continuity of
quotients.
\begin{thm}
\label{thm:HQ}Let $u$ and $v$ be strictly positive solutions $-u_{t}+Lu=0$
in $Q_{\Omega}=\Omega\times(0,\infty)$, vanishing on $Q_{2r}(Y_{0})\cap\partial_{x}Q_{\Omega}$,
where $Y_{0}=(y_{0},s_{0})\in\partial_{x}Q_{\Omega}=\partial\Omega\times(0,\infty)$
and $s_{0}\geq4r^{2}>0$. Then $\frac{u}{v}$ is H\"older continuous
in $\overline{Q_{\Omega}\cap Q_{r}(Y_{0})}$.
\end{thm}

\subsection{A universal spectral gap for the elliptic problem.}

Given the ellipticity condition and estimates on the coefficients, there
is a universal gap in the spectrum of the operator $L=\sum_{ij}a_{ij}D_{ij}+\sum b_{i}D_{i}$ between the principal
eigenvalue and the rest of the eigenvalues. We first list two
results about the Harnack principle for quotients of solutions which
are helpful to show the desired spectral gap. As with the applications of interior Harnack inequality above,
the proofs are more or less independent of the specific structure
of the equations, so we will omit them. The detailed proofs are presented in \cite{HPS1},
\cite{HPS2} and \cite{FSY}. We will proceed as in \cite{HPS1}.

We consider the following problem for a linear parabolic equation.

\begin{equation}
-u_{t}+Lu=0\label{eq:612}
\end{equation}
in $\Omega\times I$, and
\[
u=0
\]
on $\partial\Omega\times I$, where $I$ is an open interval in $\mathbb{R}$
and $L=\sum_{ij}\frac{\partial^{2}}{\partial x_{i}\partial x_{j}}+\sum_{i}b_{i}\frac{\partial}{\partial x_{i}}$
with assumptions \eqref{eq:02} and \eqref{eq:drift}

\begin{thm}
\label{cor:mq}Let $u_{1}$ and $u_{2}$ be two real solutions of
the above problem \eqref{eq:612} and let $u_{1}>0$ in $Q_{\Omega}:=\Omega\times(s,\infty)$.
Then
\begin{equation}
\sup_{Q_{\Omega}}\frac{u_{2}}{u_{1}}=M(s):=\sup_{\Omega\times\{s\}}\frac{u_{2}}{u_{1}},\label{eq:613}
\end{equation}
and
\begin{equation}
\inf_{Q_{\Omega}}\frac{u_{2}}{u_{1}}=m(s):=\inf_{\Omega\times\{s\}}\frac{u_{2}}{u_{1}}.\label{eq:614}
\end{equation}
\end{thm}
\begin{thm}
\textup{\label{thm:qm}Let }$u_{1}$ and $u_{2}$ be two positive
solutions of the above problem \eqref{eq:612} in the cylinder $Q_{\Omega}:=\Omega\times(s,\infty)$,
and let $M$ and $m$ be defined as in the above corollary, then
\begin{equation}
M(t)\leq N_{1}m(t)\label{eq:615}
\end{equation}
for
\[
t\geq s+1
\]
with a constant $N_{1}$ depends on $n$, $\nu$, $S$, $r_{0}$,
$R_{0}$ and $M_{0}$.
\end{thm}
Recall the oscillation of a real function can be defined as
\begin{equation}
osc_{\Omega}f:=\sup_{x,y\in\Omega}\left|f(x)-f(y)\right|=\sup_{\Omega}f-\inf_{\Omega}f.\label{eq:616}
\end{equation}
For a complex function, we can also define the oscillation, it can
be formulated as following:
\begin{equation}
osc_{\Omega}f:=\sup_{x,y\in\Omega}\left|f(x)-f(y)\right|=\sup_{0\leq\phi\leq2\pi}osc_{\Omega}\Re(e^{i\phi}f).\label{eq:617}
\end{equation}

\begin{prop}
\label{prop:holder}$u_{1}$and $u_{2}$ be two real solutions of
the above problem \eqref{eq:612} in the cylinder $Q_{\Omega}:=\Omega\times(s,\infty)$,
and $u_{1}>0$ in $Q_{\Omega}$ but $u$ is allowed to be complex
valued. Then
\begin{equation}
\omega(t):=osc_{\Omega\times\{t\}}\frac{u}{u_{1}}\leq\omega(t)\label{eq:618}
\end{equation}
for $t\geq s$, and
\begin{equation}
\omega(t)\leq\theta_{0}\omega(s)\label{eq:619}
\end{equation}
for $t\geq s+1$, where $\theta_{0}:=1-N_{1}^{-1}\in(0,1)$ where
$N_{1}\geq1$ is from the above theorem \ref{thm:qm}.
\end{prop}
Now we consider the operator $L=\sum_{ij}a_{ij}D_{ij}+\sum b_{i}D_{i}$ with coefficients independent of
time. We will show the above results will give us the existence of
a gap in the spectrum of $L$ that only depends on constants $n$, $\nu$,
$S$, $r_{0}$, $R_{0}$ and $M_{0}$ but not on $L$ itself.
we will call this gap a universal gap. More precisely, we consider the following
eigenvalue problem:
\begin{equation}
-Lv=\lambda v\label{eq:620}
\end{equation}
 in $\Omega$, and
\[
v=0
\]
 on $\partial\Omega$. The principal eigenvalue $\lambda_{1}$ is
defined as the eigenvalue with the smallest real part. It is well-known
actually $\lambda_{1}$ is real, algebraically simple, and the associated
eigenfunction $v_{1}$ can be chosen positive. No other eigenvalue
has a positive eigenfunction and we also have $\Re(\lambda)>\lambda_{1}$
for any other eigenvalue $\lambda$. One can find details in \cite{BHV}.
\begin{thm}
\label{thm:gap}Let $\lambda_{1}$ be the principal eigenvalue of
the above eigenvalue problem and let $\lambda$ be any other eigenvalue
of it. Then
\begin{equation}
Re(\lambda)-\lambda_{1}\geq\gamma>0,\label{eq:621}
\end{equation}
where $\gamma$ is a constant only depending on constants $n$, $\nu$,
$S$, $r_{0}$, $R_{0}$ and $M_{0}$.\end{thm}
\begin{proof}
First of all, we notice that if $v(x)$ is an eigenfunction of the
eigenvalue problem associated to an eigenvalue $\lambda$, then $u(x,t):=e^{-\lambda t}v(x)$
is a solution to the parabolic problem with $I=\mathbb{R}$. Now when
$\lambda=\lambda_{1}$, $v=v_{1}$, then the function $u_{1}(x,t):=e^{-\lambda_{1}t}v_{1}(x)$
is a positive solution of the parabolic problem on $\Omega\times I$.
For $\lambda\neq\lambda_{1}$, clearly, $v$ is not a constant multiple
of $v_{1}$ so
\begin{equation}
\omega(t):=osc_{\Omega\times\{t\}}\frac{u}{u_{1}}=e^{(\lambda_{1}-Re(\lambda))t}\omega(0),\label{eq:622}
\end{equation}
where $\omega(0)=osc_{\Omega}\frac{v}{v_{1}}\neq0$. From \cite{HPS2},
we also know $\omega(0)<\infty$. Now applying the result from the above
Proposition \ref{prop:holder}, we conclude that
\begin{equation}
\omega(1):=e^{\lambda_{1}-Re(\lambda)}\omega(0)\leq\theta_{0}\omega(0),\label{eq:623}
\end{equation}
therefore
\begin{equation}
Re(\lambda)-\lambda_{1}\geq c_{0}:=-\ln\theta_{0}>0.\label{eq:624}
\end{equation}
We notice that $c_{0}$ only depends on the prescribed constants.
\end{proof}

\section{Appendix}
In the section 2, we briefly discussed a version of Alexandrov-Bakelman-Pucci-Krylov-Tso
estimate which plays an important role in this paper.
In this appendix, we prove the version of Alexandrov-Bakelman-Pucci-Krylov-Tso
estimate we used. For more general cases, one can
find details in \cite{AIN}. Again, we use the notations $u\in W_{n,\infty}^{2,1}(Q)$
which means $u_{t},\, D_{i}u,\, D_{ij}u\in\left(L_{x}^{n}L_{t}^{\infty}\right)_{loc}$.
Also we assume $S:=\intop\left[\sup_{t}\left|b(x,t)\right|\right]^{n}dx<\infty$. 

We will start with the associated version without drift. Consider
\begin{equation}
-u_{t}+\sum_{ij}a_{ij}D_{ij}u\geq f\label{eq:93}
\end{equation}
where $f\in L_{x}^{n}L_{t}^{\infty}$.

In the following arguments, we will assume $u\in C^{2,1}$ instead of
$u\in W_{n,\infty}^{2,1}$, but the results hold for $u\in W_{n,\infty}^{2,1}$
by standard approximation arguments as \cite{KT}.
\begin{lem}
	Let $u\in C^{2,1}(Q_{\Omega})$ and suppose $Q_{\Omega}=\Omega\times(0,T)$
	with the diameter of $\Omega$ is $r$. Also assume $-u_{t}+\sum_{ij}a_{ij}D_{ij}u\geq f$
	and $\sup_{\partial_{p}Q_{\Omega}}u\leq0$. Then 
	\begin{equation}
	\sup_{Q_{\Omega}}u\leq N(\nu,n)r\left\Vert f_{-}\right\Vert _{L_{x}^{n}L_{t}^{\infty}}.\label{eq:92}
	\end{equation}
	where $f_{-}$ denotes the negative part of $f$.\end{lem}
\begin{proof}
	First of all, we notice that it suffices to consider the positive
	part of the function $u$. So without loss of generality, we might
	assume $u=0$ on $\partial_{p}Q_{\Omega}$. Following \cite{KT,LE1},
	we might also assume for some $(x_{0},\tau)\in\Omega$ with $0<\tau\leq T$
	such that $M=u(x_{0},\tau)=\sup_{\Omega}u$. From the proof of Proposition
	2.1 in \cite{KT}, we obtain the following estimate, 
	\begin{equation}
	u(x_{0},\tau)^{n+1}\leq Cr^{n}\int_{A_{u}}\left|\det\left(D_{ij}u(x,t)\right)u_{t}(x,t)\right|dxdt\label{eq:94}
	\end{equation}
	where 
	
	\[
	A_{u}=\{(x,t)\in\partial_{t}\Omega\times[0,\tau):\,\exists\xi\in\mathbb{R}^{n},\, s.t.\, u(y,s)\leq u(x,t)+\xi(y-x)\,\,\forall y\in\partial_{t}\Omega,\,0\leq s\leq\tau\}
	\]
	and $C$ only depends on $n$. Also by the discussion in \cite{KT},
	for $(x,t)\in A_{u}$, $\left(D_{ij}u(x,t)\right)$ is nonpositive
	and $u_{t}(x,t)\geq0$. We have $\sum_{ij}a_{ij}D_{ij}u\geq f$, i.e.,
	\begin{equation}
	-\sum_{ij}a_{ij}D_{ij}u\leq f_{-}\label{eq:96}
	\end{equation}
	on $A_{u}$. Therefore, 
	\begin{equation}
	\left|\det\left(D_{ij}u(x,t)\right)\right|=-\det\left(D_{ij}u(x,t)\right)\leq C(\nu)\left(-\sum_{ij}a_{ij}D_{ij}u\right)^{n}\leq\left|f_{-}\right|^{n}\label{eq:97}
	\end{equation}
	for all $(x,t)\in A_{u}$. Hence 
	\begin{equation}
	u(x_{0},\tau)^{n+1}\leq Cr^{n}\int_{A_{u}}\left|f_{-}\right|^{n}u_{t}(x,t)\,dxdt\leq Cr^{n}\int_{A_{u}}\left(\sup_{t}\left|f_{-}\right|^{n}\right)u_{t}(x,t)\,dxdt.\label{eq:98}
	\end{equation}
	We project $A_{u}$ onto $\Omega$, and we denote the projected area
	in $\Omega$ as $P_{u}$. Let 
	\begin{equation}
	I_{u}(x)=\{t\in[0,\tau[,\,(x,t)\in A_{u}\}.\label{eq:99}
	\end{equation}
	Then one can write
	\begin{equation}
	u(x_{0},\tau)^{n+1}\leq Cr^{n}\int_{P_{u}}\left(\sup_{t}\left|f_{-}\right|^{n}\right)\int_{I(x)}u_{t}(x,t)dxdt.\label{eq:910}
	\end{equation}
	Let $y=x-h\frac{\xi}{\left|\xi\right|}.$ where $h\geq d(x)$ is a
	positive number so that $y$ on the boundary of $\Omega$. By
	our condition on the boundary, we have $u(y,t)=0$. Now by the definition
	of $A_{u}$,
	\[
	0\leq u(x,t)-h\left|\xi\right|.
	\]
	Hence if $(x,t)\in A_{u}$ with $\xi$ in the definition $A_{u}$,
	we obtain $\left|\xi\right|\leq\frac{u(x,t)}{d(x)}$ where $d(x)$ denotes
	the distance from $x$ to $\partial\Omega$. 
	
	In order to analyze the time integration, we must understand the
	topology of $I_{u}(x)$. From above discussion, given the condition
	$\left|\xi\right|\leq\frac{u(x,t)}{d(x)}$, suppose we pick $t_{i}\in I_{u}(x)$
	with associated $\xi_{i}$ . Then $\{t_{i}\}$ is bounded and $\left\{ \left|\xi_{i}\right|\right\} $
	is also bounded, so we can pick a subsequence $\left\{ t_{ik}\right\} $
	with $\left\{ \xi_{ik}\right\} $ so that $t_{ik}\rightarrow t_{0}$
	and $\xi_{ik}\rightarrow\xi_{0}$. From the definition of $t_{i}$
	and $\xi_{i}$, we can conclude $t_{0}\in I_{u}(x)$ since $\xi_{0}$
	satisfies the condition in the definition of $A_{u}$. So we can also
	conclude that $I_{u}(x)$ is compact and is relatively closed to $[0,\tau]$
	for all $x\in P_{u}$. 
	
	By the basic 1-dimensional topology, we know we can write $[0,\tau]\backslash I_{u}(x)$
	as a disjoint union of finite intervals $I_{j}$, and each of them
	is one of the following four forms: $[0,\alpha)$, $(\beta,\gamma)$,
	$(\delta,\tau]$ with $\alpha\leq\beta\leq\gamma\leq\tau$. Notice
	that by the definition of $A_{u}$ and $I(x)$, 
	\begin{equation}
	\int_{[0,\alpha)}u_{t}(x,t)\,dt=u(x,\alpha)-u(x,0)\geq0\label{eq:911}
	\end{equation}
	since $\alpha\in I_{u}(x)$. 
	\begin{equation}
	\int_{(\beta,\gamma)}u_{t}(x,t)\,dt=u(x,\gamma)-u(x,\beta)\geq0\label{eq:912}
	\end{equation}
	since $\gamma\in I_{u}(x)$. For each $x\in P_{u}$, only at most
	one of the intervals in the decomposition of $[0,\tau]\backslash I(x)$
	is of the form $(\delta,\tau]$. So 
	\begin{equation}
	u(x,\tau)=\int_{0}^{\tau}u_{t}(x,t)\,dt\geq\int_{I_{u}(x)}u_{t}(x,t)\,dt+u(x,\tau)-u(x,\delta).\label{eq:913}
	\end{equation}
	Therefore, 
	
	\[
	\int_{I_{u}(x)}u_{t}(x,t)\,dt\leq u(x_{0},\tau),
	\]
	which implies 
	\begin{equation}
	u(x_{0},\tau)^{n+1}\leq Cr^{n}\int_{P_{u}}\left(\sup_{t}\left|f_{-}\right|\right)^{n}u(x_{0},\tau)\,dx\label{eq:914}
	\end{equation}
	and 
	\begin{equation}
	u(x_{0},\tau)\leq Cr\left\Vert f_{-}\right\Vert _{L_{x}^{n}L_{t}^{\infty}}\leq Cr\left\Vert f\right\Vert _{L_{x}^{n}L_{t}^{\infty}}.\label{eq:915}
	\end{equation}
\end{proof}
\begin{thm}
	(Alexandrov-Bakelman-Pucci-Krylov-Tso estimate) Let $u\in C^{2,1}(Q_{\Omega})$
	and suppose $Q_{\Omega}=\Omega\times(0,T)$ with the diameter of $\Omega$
	is $r$. Also assume $-u_{t}+Lu\geq f$ and $\sup_{\partial_{p}Q_{\Omega}}u\leq0$.
	Then 
	\begin{equation}
	\sup_{Q_{\Omega}}u\leq N(\nu,n,S)r\left\Vert f\right\Vert _{L_{x}^{n}L_{t}^{\infty}}.\label{eq:91}
	\end{equation}
\end{thm}
\begin{proof}
	Again as above we assume for some $(x_{0},\tau)\in\Omega$ with $0<\tau\leq T$
	such that $M=u(x_{0},\tau)=\sup_{\Omega}u>0$.
	
	Given $-u_{t}+Lu\geq f$ with drift, we move the drift to the right
	hand side. Then by Cauchy-Schwarz inequality and H\"older's inequality, we obtain for a fixed constant $\mu\neq0$ to be determined later
	\begin{eqnarray}
	u_{t}-\sum_{ij}a_{ij}D_{ij}u & \leq & f_{-}+\sum_{i}b_{i}D_{i}u\nonumber \\
	& \leq & (\vec{b},\,\mu^{-1}f_{-})\cdot(\left|\nabla u\right|,\mu)\nonumber \\
	& \leq & \left(\left|\sup_{t}b\right|^{n}+\left(\mu^{-1}\sup_{t}f_{-}\right)^{n}\right)^{\frac{1}{n}}\left(\left|\nabla u\right|^{n}+\nu\mu^{n}\right)^{\frac{1}{n}}(1+1)^{\frac{n-2}{n}}\label{eq:916}
	\end{eqnarray}
	We consider 
	
	\begin{equation}
	D=\left\{ (\xi,h):\,\left|\xi\right|\leq M/r,\, r\left|\xi\right|<h<M\right\} \label{eq:917}
	\end{equation}
	following the proof of Proposition 2.1 in \cite{KT}. We know if $g\in C(\mathbb{R}^{n+1})$
	is nonnegative, we have
	\begin{equation}
	\int_{D}g(\xi,h)\,d\xi dh\leq\int_{A_{u}}g(\nabla u,u_{t})\left|\det\left(D_{ij}u(x,t)u_{t}(x,t)\right)\right|dxdt.\label{eq:919}
	\end{equation}
	Take 
	\begin{equation}
	g(\xi,h)=\left(\left|\xi\right|^{n}+\mu^{n}\right)^{-1}.\label{eq:918}
	\end{equation}
	Then the left hand side of \eqref{eq:919} is 
	\begin{equation}
	\int_{D}g(\xi,h)\,d\xi dh=C\int_{0}^{M/r}(M-kr)k^{n-1}\left(k^{n}+\mu^{n}\right)^{-1}dk.\label{eq:920}
	\end{equation}
	For the right hand side of \eqref{eq:919}, by a similar argument as the above theorem, we can conclude
	
	\begin{equation}
	\int_{A_{u}}\left|f_{-}\right|^{n}u_{t}(x,t)\,dxdt\leq C\int_{A_{u}}\left(\sup_{t}\left|f_{-}\right|^{n}\right)u_{t}(x,t)\,dxdt.\label{eq:921}
	\end{equation}

	\begin{equation}
	\int_{A_{u}}g(\nabla u,u_{t})\left|\det\left(D_{ij}u(x,t)u_{t}(x,t)\right)\right|dxdt\leq C\int_{P_{u}}\left(\left|\sup_{t}b\right|^{n}+\mu^{-n}\left(\sup_{t}f_{-}\right)^{n}\right)M\,dx.\label{eq:922}
	\end{equation}
	So we need to calculate 
	\begin{equation}
	\int_{0}^{M/r}(M-kr)k^{n-1}\left(k^{n}+u^{n}\right)^{-1}dk.\label{eq:923}
	\end{equation}
	Notice that
	\begin{equation}
	(kr)k^{n-1}\left(k^{n}+u^{n}\right)^{-1}\leq r,\label{eq:925}
	\end{equation}
	\begin{equation}
	\int_{0}^{M/r}Mk^{n-1}\left(k^{n}+\mu^{n}\right)^{-1}dk=M\log\left(\left(\frac{M}{r\mu}\right)^{n}+1\right).\label{eq:926}
	\end{equation}
	So we can bound the left hand side of equation \eqref{eq:922} from
	below, 
	\begin{equation}
	M\log\left(\left(\frac{M}{r\mu}\right)^{n}+1\right)-M\leq\int_{0}^{M/r}(M-kr)k^{n-1}\left(k^{n}+u^{n}\right)^{-1}dk.\label{eq:927}
	\end{equation}
	Therefore, 
	\begin{equation}
	M\log\left(\left(\frac{M}{r\mu}\right)^{n}+1\right)-M\leq C\int\left(\left|\sup_{t}b\right|^{n}+\mu^{-n}\left(\sup_{t}f_{-}\right)^{n}\right)Mdx.\label{eq:928}
	\end{equation}
	Since $M>0$, and if  we take $\mu=\int\left(\sup_{t}f_{-}\right)$, one can conclude
	\begin{equation}
	\log\left(\left(\frac{M}{r\left\Vert f_{-}\right\Vert _{L_{x}^{n}L_{t}^{\infty}}}\right)^{n}+1\right)\leq C\int\left(\left|\sup_{t}b\right|^{n}+C_{1}\right)dx,\label{eq:929}
	\end{equation}
	where $C$ and $C_{1}$ only depend on $\nu$ and $n$. Finally,
	we exponentiate both sides to obtain
	\begin{equation}
	M\leq N\left(n,\nu,\int\left|\sup_{t}b\right|^{n}\right)r\left\Vert f_{-}\right\Vert _{L_{x}^{n}L_{t}^{\infty}}.\label{eq:930}
	\end{equation}
\end{proof}

\section{Acknowledgment}
This work was initiated when I was still at the University of Minnesota Twin Cities. I would like to thank Professor Mikhail Safonov for suggesting this interesting problem to me, and for many motivating discussions. I also want to thank Professor Luis Silvestre for discussions on the supercritical scaling case and some related comments. Finally, I would like to thank Hao Jia, Tianling Jin, Jon Rubin, and Ben Seeger for their kind encouragement.

\end{document}